\documentclass[12pt]{article}
\usepackage{amsfonts}
\usepackage{mathrsfs,dsfont,bbm}

\usepackage[latin1]{inputenc}
\usepackage{amsmath,amssymb}
\usepackage{latexsym}
\usepackage{epstopdf}
\usepackage{geometry}
\usepackage[active]{srcltx}
\usepackage[
bookmarks=true,         
bookmarksnumbered=true, 
colorlinks=true, pdfstartview=FitV, linkcolor=blue, citecolor=blue,
urlcolor=blue]{hyperref}

\newtheorem{theorem}{Theorem}[section]

\newtheorem{lemma}[theorem]{Lemma}

\newtheorem{definition}[theorem]{Definition}
\newtheorem{assumption}[theorem]{Assumption}
\newtheorem{remark}{Remark}[section]
\newtheorem{example}[theorem]{Example}
\numberwithin{equation}{section}
\newenvironment{proof}[1][Proof]{\noindent\textit{#1.} }{\hfill \rule{0.5em}{0.5em}}

\begin{document}
\title{Stochastic averaging principle and stability for multi-valued McKean-Vlasov stochastic differential equations with jumps}
\date{\today}

\author
{Guangjun Shen, Jie Xiang and Jiang-Lun Wu
\thanks{Corresponding author. This research is supported by the  National Natural Science Foundation of China (12071003)} }

\maketitle

\begin{abstract}
In this paper, we consider the stochastic averaging principle and stability for multi-valued McKean-Vlasov stochastic differential equations with jumps. First, under certain averaging conditions, we are able to show that the solutions of the equations concerned can be approximated by solutions of the associated averaged multi-valued McKean-Vlasov stochastic differential equations with jumps in the sense of the mean square convergence. Second, we extend the classical It\^{o}'s formula from stochastic differential equations to multi-valued McKean-Vlasov stochastic differential equations with jumps. Last, as  application of It\^{o}'s formula, we present the exponential stability of second moments, the exponentially 2-ultimate boundedness and the almost surely asymptotic stability for their solutions in terms of a Lyapunov function.
\end{abstract}

\section{Introduction}\label{sec1}
As is well known, multi-valued stochastic differential equations are widely applied to
model stochastic systems in different branches of science and industry. Averaging principle, stability, boundedness and applications of the solution are the most popular research topic in the field of stochastic dynamic systems and control.
In this paper, we consider  the following multi-valued McKean-Vlasov stochastic differential equations (MVSDEs) with jumps:
\begin{equation}\label{sec1-eq1.1}
\begin{split}
dX_t\in&-\mathcal{A}(X_t)dt+b(t,X_t,\mathscr{L}_{X_t})dt+\sigma(t,X_t,\mathscr{L}_{X_t})dB_t\\
&\quad+\int_{U_0}f(t,X_{t-},\mathscr{L}_{X_t},u)\widetilde{N}(dt,du),\quad t\in[0,T],
\end{split}
\end{equation}
where $\mathscr{L}_{X_t}$ denotes the probability distribution of $X_t,$
$\mathcal{A}:\mathbb{R}^d\rightarrow2^{\mathbb{R}^d}$ is a multi-valued maximal monotone operator and $\hbox{Int}(D(\mathcal{A}))\neq \emptyset$ ($\hbox{Int}(D(\mathcal{A}))$ denotes the interior of the  domain $D(\mathcal{A})$), $B$ is an $m$-dimensional Brownian motion, $N$ is the counting measure
of a stationary Poisson point process with characteristic measure $v$ on some measurable
space $(\mathbb{U},\mathscr{U}),$ $\widetilde{N}(dt,du)$ stands for the compensated martingale measure of $N(dt,du)$ with $v(\mathbb{U}\verb|\|\mathbb{U}_0)<\infty, \mathbb{U}_0\in\mathscr{U}.$
The coefficients
$b:[0,T]\times\mathbb{R}^d\times\mathcal{M}_2(\mathbb{R}^d)\rightarrow\mathbb{R}^d$, $\sigma:[0,T]\times\mathbb{R}^d\times\mathcal{M}_2(\mathbb{R}^d)\rightarrow\mathbb{R}^{d\times m}$ and
$f:[0,T]\times\mathbb{R}^d\times\mathcal{M}_2(\mathbb{R}^d)\times\mathbb{U}_0\rightarrow\mathbb{R}^d$
are all Borel measurable (see Section \ref{sec2}). It is well
known (see C\'{e}pa \cite{Cepa1}) that these equations include stochastic variational inequalities as a special case where the
maximal monotone operator $\mathcal{A}$ is the subdifferential operator of some convex function.
It is worthwhile mentioning here that stochastic variational inequalities generalize reflected stochastic differential equations in a convex domain.

When $\mathcal{A}=0$, if we ignore random measure $\widetilde{N},$  Equation \eqref{sec1-eq1.1} becomes  McKean-Vlasov stochastic differential equations (MVSDEs). Generally, nonlinear Fokker-Planck equations can be characterised by MVSDEs, which are also named as distribution dependent SDEs or mean field SDEs. A distinct feature of such
systems is the appearance of probability laws in the coefficients of the resulting equations. There has been an increasing interest to study existence and uniqueness for solutions of MVSDEs. Wang \cite{Wang} established strong well-posedness of MVSDEs with one-sided Lipschitz continuous drifts and Lipschitz-continuous dispersion coefficients. Under integrability conditions on distribution dependent coefficients, Huang and Wang \cite{Huang1} obtained the existence and uniqueness for MVSDEs with non-degenerate noise.  Mehri and Stannat \cite{Mehri} proposed a Lyapunov-type approach to the problem of existence and uniqueness of general law-dependent SDEs. Many interesting studies of MVSDEs have been developed further in Bao et al. \cite{Bao}, Huang et al. \cite{Huang2}, Ren and Wang \cite{Ren}, R\"{o}ckner and Zhang \cite{Rockner}, Mishura and Veretennikov \cite{Mishura}, Chaudru de Raynal \cite{Chaudru}, Hammersley et al \cite{Hammersley}, Shen et al.  \cite{G.-J. Shen} and  the references therein. However, when the noise is the random measure $\widetilde{N},$ the results about the existence and uniqueness for solutions of MVSDEs are  rarely. Let us mention some related works.  Qiao and Wu \cite{H. Qiao} proved that the path independent property of additive functionals of MVSDEs with jumps can be  characterized by nonlinear partial integro-differential
equations involving $L$-derivatives with respect to probability measures introduced
by Lions. Liu et al. \cite{W. Liu} established large and moderate deviation principles for MVSDEs with jumps by applying the weak convergence method.

When $\mathcal{A}\neq0$, in the distribution independent case, there have been many fundamental studies on  Multi-valued stochastic differential equations  driven Brownian motion. Wu \cite{Wu} obtained the  existence conditions of weak solutions to multi-valued
stochastic differential equations with discontinuous coefficients.
Ren, Xu and Zhang \cite{Ren1} studied large deviation principle of Freidlin-Wentzell type for multi-valued stochastic differential equations with monotone drifts that in particular contain
a class of SDEs with reflection in a convex domain. Ren, Wu and Zhang \cite{Ren2} proved a large deviation principle of Freidlin-Wentzell type for multi-valued stochastic differential equations that is a little more general than the results obatined by Ren et al.\cite{Ren1}. As an application, they derived a functional iterated logarithm law for the solutions. Zhang \cite{H.Zhang} proved moderate deviation principle for multi-valued stochastic differential equations using weak convergence approach.
Stochastic averaging principle for Multi-valued SDEs driven by Brownian motion was obtained by  Ngoran and  Modeste \cite{Goran}, Xu and Liu \cite{Xu}.
On the other hand,  there has been an increasing interest to study the  Multi-valued stochastic differential equations  driven jump process. Ren and Wu \cite{Ren3} proved the existence and uniqueness of solutions of multi-valued stochastic differential equations driven by Poisson point processes.  Guan and Wu \cite{Guan} studied the exponential ergodicity of diffusions generated by a multi-valued stochastic differential equation with L\'{e}vy jumps. Guo and Pei \cite{Guo} obtained stochastic averaging principles for multi-valued stochastic differential equations driven by Poisson point processes. Mao et al. \cite{W. Mao} studied the averaging principle for multivalued
SDEs with jumps and non-Lipschitz coefficients. More recently, Gong and Qiao \cite{J. Gong}  concerned the stability for  multi-valued MVSDEs driven by Brownian motion with non-Lipschitz coefficients. Fang et al. \cite{K. Fang} presented sufficient conditions and criteria to establish the large and
moderate deviation principle for multi-valued MVSDEs by
means of the weak convergence method.  The authors of this paper obtained stochastic averaging principle for multi-valued MVSDEs driven by Brownian motion \cite{G.-J. Shen2}.
However, to the best of our knowledge, there are no literature study  multi-valued MVSDEs  with jumps \eqref{sec1-eq1.1}.
In order to fill the gap, the objective of this paper  are two folds. Firstly,  under certain averaging condition, we  show that the solutions of the equations concerned can be approximated by solutions of the associated averaged multi-valued MVSDEs with jumps in the sense of the mean square convergence. Secondly, we extend the classical It\^{o}'s formula from stochastic differential equations to multi-valued MVSDEs with jumps. As  application of It\^{o}'s formula, we present the exponential stability of second moments, the exponentially 2-ultimate boundedness and the almost surely asymptotic stability for their solutions in terms of a Lyapunov function.

The rest of the paper is organized as follows. In Section 2, we recall some basic notation and introduce maximal monotone operators and derivatives for functions on $\mathcal{M}_2(\mathbb{R}^d)$. In Section 3, we prove stochastic averaging principle for the equation concerned. In Section 4, we obtain the It\^{o}'s formula for multi-valued MVSDEs with jumps. Then in Section 5 we present the exponential stability of second moments, the exponentially 2-ultimate boundedness and the almost surely asymptotic stability for the strong solutions of Equation \eqref{sec1-eq1.1}.

\section{Preliminaries}\label{sec2}

\subsection{Notations}\label{sec2.1}
We use $|\cdot|$ and $\|\cdot\|$ for norms of vectors and matrices, respectively. Let $\langle\cdot,\cdot\rangle$ denote the scalar product in $\mathbb{R}^d$ and let $B^*$ denote the transpose of the matrix $B$. Let $C(\mathbb{R}^d)$ be the collection of continuous functions on $\mathbb{R}^d$ and $C^2(\mathbb{R}^d)$ be the space of continuous functions on $\mathbb{R}^d$
which have continuous partial derivatives of order up to 2. Define the Banach space
$$
C_\rho(\mathbb{R}^d):=\Big\{\varphi\in C(\mathbb{R}^d),
\|\varphi\|_{C_\rho(\mathbb{R}^d)}:=\sup_{x\in\mathbb{R}^d}\frac{|\varphi(x)|}{(1+|x|)^2}
+\sup_{x\neq y}\frac{|\varphi(x)-\varphi(y)|}{|x-y|}<\infty\Big\}.
$$
Let $\mathscr{B}(\mathbb{R}^d)$ be the Borel $\sigma$-algebra on $\mathbb{R}^d$ and $\mathscr{P}(\mathbb{R}^d)$ be the space of all probability measures defined on $\mathscr{B}(\mathbb{R}^d)$ carrying the usual topology of weak convergence. Let $\mathcal{M}_2(\mathbb{R}^d)$ be the set of probability measures on
$\mathscr{B}(\mathbb{R}^d)$ with finite second order moments. That is,
$$
\mathcal{M}_2(\mathbb{R}^d):=\Big\{\mu\in\mathscr{P}(\mathbb{R}^d):\|\mu\|_2^2:=\int_{\mathbb{R}^d}|x|^2\mu(dx)<\infty\Big\}.
$$
Define the following metric on $\mathcal{M}_2(\mathbb{R}^d)$:
$$
\rho(\mu,\nu):=\sup_{\|\varphi\|_{C_\rho(\mathbb{R}^d)}\leq1}
\Big|\int_{\mathbb{R}^d}\varphi(x)\mu(dx)-\int_{\mathbb{R}^d}\varphi(x)\nu(dx)\Big|,\quad \mu,\nu\in\mathcal{M}_2(\mathbb{R}^d).
$$
It is clear that $(\mathcal{M}_2(\mathbb{R}^d),\rho)$ is a complete metric space and $\rho(\mathscr{L}_X,\mathscr{L}_Y)\leq(\mathbb{E}|X-Y|^2)^{\frac12}$, and further that the convergence with respect to the metric $\rho$ is equivalent to the weak convergence
(see e.g. Gong and Qiao \cite{J. Gong}).

\subsection{Multi-valued MVSDEs with jumps}\label{sec2.2}
In this subsection, we introduce multi-valued MVSDEs with jumps.
Let $(\Omega,\mathscr{F},\mathbb{P})$ be a complete probability space, $(B_t)_{t\geq0}$ be a $m$-dimensional Brownian motion, $\mathcal{D}([0,T];\mathbb{R}^d)$ denote the family of all c\`{a}dl\`{a}g functions, $(\mathbb{U},\|\cdot\|_{\mathbb{U}})$ be a finite dimensional normed space with its Borel $\sigma$-algebra $\mathscr{U}$, $v$ be a $\sigma$-finite measure defined on $(\mathbb{U},\mathscr{U})$. We fix $\mathbb{U}_0=\{\|u\|_{\mathbb{U}}\leq\alpha\}$, where $\alpha>0$ is a constant, with $v(\mathbb{U}\verb|\|\mathbb{U}_0)<\infty$ and $\int_{\mathbb{U}_0}\|u\|_{\mathbb{U}}^2v(du)<\infty$. Let $p=p(t)$ be a stationary Poisson point process with characteristic $v$,
$D_{p(t)}$ be a countable subset of $\mathbb{R}_+$. Denote by $N(dt,du)$ the Poisson counting measure associated with $p(t)$, i.e.,
$$
N(t,A):=\sum_{s\in D_{p(s)},s\leq t}I_A(p(t)),
$$
with intensity $\mathbb{E}N(dt,du)=v(du)dt$. Denote
$
\widetilde{N}(dt,du):=N(dt,du)-v(du)dt,
$
that is, $\widetilde{N}(dt,du)$ stands for the compensated martingale measure of $N(dt,du)$. Moreover, $B_.$ and $N(dt,du)$ are mutually independent. Fix $T>0$ and let $\{\mathscr{F}_t\}_{t\in[0,T]}$ be the filtration generated by $(B_t)_{t\geq0}$ and $N(dt,du)$, and augmented by a $\sigma$-field $\mathscr{F}^0$, i.e.
\begin{align*}
 \mathscr{F}^0_t:=\sigma\{B_s,N((0,s],A):0\leq s\leq t,A\in\mathscr{U}\}, ~~~\mathscr{F}_t:=(\cap_{s>t}\mathscr{F}^0_s)\vee\mathscr{F}^0,\quad t\in[0,T],
\end{align*}
where $\mathscr{F}^0\subset\mathscr{F}$ has the following properties:
\begin{itemize}
\item [(i)] $(B_t)_{t\geq0}$ and $N(dt,du)$ are independent of $\mathscr{F}^0$;
\item [(ii)] $\mathcal{M}_2(\mathbb{R}^d)=\{\mathbb{P}\circ\xi^{-1},\xi\in
    L^2(\mathscr{F}^0;\mathbb{R}^d)\}$;
\item [(iii)] $\mathscr{F}^0\supset\mathcal{N}$ and $\mathcal{N}$ is the collection of all $\mathbb{P}$-null sets.
\end{itemize}

\begin{lemma}\label{sec2-lem2.1}(\cite{W. Mao})
Let $\phi:\mathbb{R}_+\times\mathbb{U}_0\rightarrow\mathbb{R}^d$ and assume that $\mathbb{E}\Big(\int_0^t\int_{\mathbb{U}_0}|\phi(s,u)|^2v(du)ds\Big)<\infty$, then
$$
\mathbb{E}\int_0^t\int_{\mathbb{U}_0}|\phi(s,u)|^2N(ds,du)=\mathbb{E}\int_0^t\int_{\mathbb{U}_0}|\phi(s,u)|^2v(du)ds.
$$
Moreover, there exists a positive constant $C$ such that
$$
\mathbb{E}\Big(\sup_{0\leq t\leq T}\Big|\int_0^t\int_{\mathbb{U}_0}\phi(s,u)\widetilde{N}(ds,du)\Big|\Big)
\leq C\mathbb{E}\Big(\int_0^T\int_{\mathbb{U}_0}|\phi(s,u)|^2N(ds,du)\Big)^{\frac12}.
$$
\end{lemma}

\subsection{Maximal monotone operators}\label{sec2.3}
In the subsection, we introduce the basic definitions and properties of the multi-valued maximal monotone operators. For more details about the maximal monotone operator, we refer to (\cite{Cepa},\cite{X.Zhang},\cite{J. Gong}).

Denote by $2^{\mathbb{R}^d}$ the set of all subsets of $\mathbb{R}^d$, a map $\mathcal{A}:\mathbb{R}^d\rightarrow2^{\mathbb{R}^d}$ is called a multi-valued operator on $\mathbb{R}^d$. Given such a multi-valued operator $\mathcal{A}$, define:
$$
D(\mathcal{A})=\{x\in{\mathbb{R}^d}:\mathcal{A}(x)\neq\emptyset\},~~
Gr(\mathcal{A})=\{(x,y)\in{\mathbb{R}^d}\times{\mathbb{R}^d}:x\in{\mathbb{R}^d},y\in\mathcal{A}(x)\}.
$$
$\mathcal{A}^{-1}$ is defined by $
y\in\mathcal{A}^{-1}(x)\Leftrightarrow x\in\mathcal{A}(y).
$

\begin{definition}\label{sec2-def2.2}
\begin{itemize}
\item [(i)] A multi-valued operator $\mathcal{A}$ is called monotone if
$$
\langle x_1-x_2,y_1-y_2\rangle\geq0,\quad \forall(x_1,y_1),(x_2,y_2)\in Gr(\mathcal{A}).
$$
\item [(ii)] A monotone operator $\mathcal{A}$ is called maximal monotone if and only if
$$
(x_1,y_1)\in Gr(\mathcal{A})\Leftrightarrow\langle x_1-x_2,y_1-y_2\rangle\geq0,\quad\forall(x_2,y_2)\in Gr(\mathcal{A}).
$$
\end{itemize}
\end{definition}

Given $T>0$. Let $\mathscr{V}_0$ be the set of all continuous functions $K:[0,T]\rightarrow\mathbb{R}^d$ with finite variations and $K_0=0$. For $K\in\mathscr{V}_0$ and $s\in[0,T]$, we shall use $|K|^s_0$ to denote the variation of $K$ on $[0,s]$. Set
\begin{align*}
\mathscr{A}:=\Big\{&(X,K):X\in \mathcal{D}([0,T],\overline{D(\mathcal{A})}),K\in\mathscr{V}_0,
~~\hbox{and}~~\langle X_t-\alpha_t,dK_t-\beta_tdt\rangle\geq0~~\hbox{a.s.}\\
&~~~~~~~~~~~~~~\hbox{for any}~~
\alpha, \beta\in\mathcal{D}([0,T],\overline{D(\mathcal{A})})~~\hbox{satisfying} ~~(\alpha_t,\beta_t)\in Gr(\mathcal{A})
\Big\}.
\end{align*}

\begin{lemma}\label{sec2-lem2.3}(\cite{Cepa1})
Suppose $(X,K)$ and $(\tilde{X},\tilde{K})$ are two pairs in $\mathscr{A}$, then
$$
\langle X_t-\tilde{X}_t,dK_t-d\tilde{K}_t\rangle\geq0.
$$
\end{lemma}

\begin{lemma}\label{sec2-lem2.4}(\cite{Cepa})
Let $\mathcal{A}$ be a maximal monotone operator on $\mathbb{R}^d$, then, for each $x\in D(\mathcal{A})$, $\mathcal{A}(x)$ is a closed and convex subset of $\mathbb{R}^d$. Let $\mathcal{A}^\circ(x):=proj_{\mathcal{A}(x)}(0)$ be the minimal section of $\mathcal{A}$, where $proj_{D}$ is designated as the projection on every closed and convex subset $D$ on $\mathbb{R}^d$ and $proj_{\emptyset}(0)=\infty$. Then
$$
x\in D(\mathcal{A})\Leftrightarrow |\mathcal{A}^\circ(x)|<\infty.
$$
\end{lemma}

\subsection{Derivatives for functions on $\mathcal{M}_2(\mathbb{R}^d)$}\label{sec2.4}
In the rest of the section, we recall the definition of $L$-derivative for functions on $\mathcal{M}_2(\mathbb{R}^d)$. The definition was first introduced by Lions \cite{Lions}. Moreover, he used some abstract probability spaces to describe the $L$-derivatives. Here, for the convenience to understand the definition, we apply a straight way to state it (Ref. \cite{P.Ren}). Let $I$ be the identify map on $\mathbb{R}^d$. For $\mu\in\mathcal{M}_2(\mathbb{R}^d)$ and $\phi\in L^2(\mathbb{R}^d,\mathscr{B}(\mathbb{R}^d),\mu;\mathbb{R}^d)$, $\mu(\phi):=\int_{\mathbb{R}^d}\phi(x)\mu(dx)$ and  $\mu\circ(I+\phi)^{-1}\in\mathcal{M}_2(\mathbb{R}^d).$

\begin{definition}\label{sec2-def2.5}
\begin{itemize}
\item [(i)] A function $h:\mathcal{M}_2(\mathbb{R}^d)\rightarrow\mathbb{R}$ is called $L$-differentiable at $\mu\in\mathcal{M}_2(\mathbb{R}^d)$, if the functional
$$
L^2(\mathbb{R}^d,\mathscr{B}(\mathbb{R}^d),\mu;\mathbb{R}^d)\ni\phi\rightarrow h(\mu\circ(I+\phi)^{-1})
$$
is Fr\'{e}chet differentiable at $0\in L^2(\mathbb{R}^d,\mathscr{B}(\mathbb{R}^d),\mu;\mathbb{R}^d)$; that is, there exists a unique $\xi\in L^2(\mathbb{R}^d,\mathscr{B}(\mathbb{R}^d),\mu;\mathbb{R}^d)$ such that
$$
\lim_{\mu(|\phi|^2)\rightarrow0}\frac{h(\mu\circ(I+\phi)^{-1})-h(\mu)-\mu(\langle\xi,\phi\rangle)}{\sqrt{\mu(|\phi|^2)}}=0.
$$
In the case, we denote $\partial_\mu h(\mu)=\xi$ and call it the $L$-derivative of $h$ at $\mu$.
\item [(ii)] A function $h:\mathcal{M}_2(\mathbb{R}^d)\rightarrow\mathbb{R}$ is called $L$-differentiable on $\mathcal{M}_2(\mathbb{R}^d)$ if $L$-derivative $\partial_\mu h(\mu)$ exists for all $\mu\in\mathcal{M}_2(\mathbb{R}^d)$.
\item [(iii)] By the same way, $\partial_\mu^2h(\mu)(y,y^{'})$ for $y,y^{'}\in\mathbb{R}^d$ can be defined.
\end{itemize}
\end{definition}

Next, we introduce some related spaces.
\begin{definition}\label{sec2-def2.6}
The function $h$ is said to be in $C^2(\mathcal{M}_2(\mathbb{R}^d))$, if $\partial_\mu h$ is continuous, for any $\mu\in\mathcal{M}_2(\mathbb{R}^d)$, $\partial_\mu h(\mu)(\cdot)$ is differentiable, and its derivative $\partial y\partial_\mu h:\mathcal{M}_2(\mathbb{R}^d)\times\mathbb{R}^d\rightarrow\mathbb{R}^d\otimes\mathbb{R}^d$ is continuous, and for any $y\in\mathbb{R}^d$, $\partial_\mu h(\cdot)(y)$ is differentiable, and its derivative $\partial_\mu^2h:\mathcal{M}_2(\mathbb{R}^d)\times\mathbb{R}^d\times\mathbb{R}^d\rightarrow\mathbb{R}^d\otimes\mathbb{R}^d$
is continuous.
\end{definition}

\begin{definition}\label{sec2-def2.7}
\begin{itemize}
\item [(i)] The function $h:[0,T]\times\mathbb{R}^d\times\mathcal{M}_2(\mathbb{R}^d)\rightarrow\mathbb{R}$ is said to be in $C^{1,2,2}([0,T]\times\mathbb{R}^d\times\mathcal{M}_2(\mathbb{R}^d))$, if $h(t,x,\mu)$ is $C^1$ in $t\in[0,T]$, $C^2$ in $x\in\mathbb{R}^d$ and $\mu\in\mathcal{M}_2(\mathbb{R}^d)$ respectively, and its derivatives
$$
\partial_th(t,x,\mu),\quad \partial_xh(t,x,\mu),\quad \partial_x^2h(t,x,\mu),\quad \partial_\mu h(t,x,\mu)(y),
$$
$$
\partial y\partial_\mu h(t,x,\mu)(y),\quad \partial_\mu^2 h(t,x,\mu)(y,y^{'})
$$
are jointly continuous in the corresponding variable family $(t,x,\mu)$, $(t,x,\mu,y)$ or $(t,x,\mu,y,y^{'})$.
\item [(ii)] The function $h:[0,T]\times\mathbb{R}^d\times\mathcal{M}_2(\mathbb{R}^d)\rightarrow\mathbb{R}$ is said to be in $C^{1,2,2}_b([0,T]\times\mathbb{R}^d\times\mathcal{M}_2(\mathbb{R}^d))$, if $h\in C^{1,2,2}([0,T]\times\mathbb{R}^d\times\mathcal{M}_2(\mathbb{R}^d))$ and all its derivatives are uniformly bounded on $[0,T]\times\mathbb{R}^d\times\mathcal{M}_2(\mathbb{R}^d)$. If $h\in C^{1,2,2}([0,T]\times\mathbb{R}^d\times\mathcal{M}_2(\mathbb{R}^d))$ or $h\in C^{1,2,2}_b([0,T]\times\mathbb{R}^d\times\mathcal{M}_2(\mathbb{R}^d))$ and $h$ is independent of $t$, we write $h\in C^{2,2}(\mathbb{R}^d\times\mathcal{M}_2(\mathbb{R}^d))$ or $h\in C^{2,2}_b(\mathbb{R}^d\times\mathcal{M}_2(\mathbb{R}^d))$.
\item [(iii)] The function $h:[0,T]\times\mathbb{R}^d\times\mathcal{M}_2(\mathbb{R}^d)\rightarrow\mathbb{R}$ is said to be in $C^{1,2,2;1}_b([0,T]\times\mathbb{R}^d\times\mathcal{M}_2(\mathbb{R}^d))$, if $h\in C^{1,2,2}_b([0,T]\times\mathbb{R}^d\times\mathcal{M}_2(\mathbb{R}^d))$ and all its derivatives are Lipschitz continuous. In addition, if $h$ is independent of $t$, we write $h\in C^{2,2;1}_b(\mathbb{R}^d\times\mathcal{M}_2(\mathbb{R}^d))$.
\item [(iv)] The function $h:[0,T]\times\mathbb{R}^d\times\mathcal{M}_2(\mathbb{R}^d)\rightarrow\mathbb{R}$ is said to be in $C^{1,2,2}_{b,+}([0,T]\times\mathbb{R}^d\times\mathcal{M}_2(\mathbb{R}^d))$, if $h\in C^{1,2,2}_b([0,T]\times\mathbb{R}^d\times\mathcal{M}_2(\mathbb{R}^d))$ and $h\geq0$. If $h\in C^{1,2,2;1}_b([0,T]\times\mathbb{R}^d\times\mathcal{M}_2(\mathbb{R}^d))$ and $h\geq0$, we write If $h\in C^{1,2,2;1}_{b,+}([0,T]\times\mathbb{R}^d\times\mathcal{M}_2(\mathbb{R}^d))$.
\end{itemize}
\end{definition}

Now, we consider the multi-valued MVSDEs with jumps \eqref{sec1-eq1.1}.

\begin{definition}\label{sec2-def2.8}
For any $T>0$, a pair of $\mathscr{F}_t$-adapted processes $(X_.,K_.)$ is called a strong solution of Equation \eqref{sec1-eq1.1} with the initial value $\xi$ if the pair $(X_.,K_.)$
on a filtered probability space $(\Omega,\mathscr{F},\{\mathscr{F}_t\}_{t\in[0,T]},\mathbb{P})$
such that
\begin{itemize}
\item [(i)] $\mathbb{P}(X_0=\xi)=1$,
\item [(ii)] $(X_.(\omega),K_.(\omega))\in\mathscr{A}$ a.s. $\mathbb{P}$,
\item [(iii)] it holds that
$$
\mathbb{P}\Big\{\int_0^T\Big(|b(s,X_s,\mathscr{L}_{X_s})|+\|\sigma(s,X_s,\mathscr{L}_{X_s})\|^2
+\int_{\mathbb{U}_0}|f(s,X_s,\mathscr{L}_{X_s},u)|^2v(du)\Big)ds<+\infty\Big\}=1,
$$
and
\begin{align*}
X_t=\xi-K_t&+\int_0^tb(s,X_s,\mathscr{L}_{X_s})ds+\int_0^t\sigma(s,X_s,\mathscr{L}_{X_s})dB_s\\
&+\int_{[0,t]\times U_0}f(s,X_{s-},\mathscr{L}_{X_s},u)\widetilde{N}(ds,du),\quad t\in[0,T]\quad a.s..
\end{align*}
\end{itemize}
\end{definition}

In the following, we give some conditions to assure the existence and uniqueness of strong solutions for Equation \eqref{sec1-eq1.1}. We make the following assumptions:

\begin{assumption}\label{sec2-ass2.9}
\begin{itemize}
\item[(i)]
{
For any $x,y\in\mathbb{R}^d$, $\mu,\nu\in\mathcal{M}_2(\mathbb{R}^d)$ and $t\in[0,T]$, $u\in\mathbb{U}_0$, there exists an non-decreasing bounded function $L_1:[0,\infty)\rightarrow(0,\infty)$ such that
$$
|b(t,x,\mu)-b(t,y,\nu)|^2+\|\sigma(t,x,\mu)-\sigma(t,y,\nu)\|^2\leq L_1(t)\kappa(\beta|x-y|^2+\rho^2(\mu,\nu)),
$$
$$
|f(t,x,\mu,u)-f(t,y,\nu,u)|^2\leq L_1(t)\|u\|^2_{\mathbb{U}}\varphi(\beta|x-y|^2+\rho^2(\mu,\nu)).
$$
where $\beta$ is a positive constant,
$\kappa(\cdot),\varphi(\cdot)$ are two concave nondecreasing functions such that $\kappa(0)=\varphi(0)=0$, and
$$
\int_{0^+}\frac{1}{\kappa(u)+\varphi(u)+u}du=\infty.
$$
}
\item[(ii)]
{
For any $x,y\in\mathbb{R}^d$, $\mu,\nu\in\mathcal{M}_2(\mathbb{R}^d)$ and $t\in[0,T]$, $u\in\mathbb{U}_0$, there exists an non-decreasing bounded function $L_2:[0,\infty)\rightarrow(0,\infty)$ such that
$$
|b(t,x,\mu)|^2+\|\sigma(t,x,\mu)\|^2\leq L_2(t)(1+|x|^2+\|\mu\|_2^2),
$$
$$
|f(t,x,\mu,u)|^2\leq L_2(t)\|u\|^2_{\mathbb{U}}(1+|x|^2+\|\mu\|_2^2).
$$
}
\end{itemize}
\end{assumption}

\begin{example}\label{sec2-ex2.10}
Let $L>0$ and $\delta\in(0,1/e)$ be sufficiently small. Define\\
$\kappa_1(u)=\varphi_1(u)=Lu,u\geq 0.$\\
$\kappa_2(u)=\varphi_2(u)=\left\{
  \begin{array}{ll}
    u\log(u^{-1}), & {0\leq u\leq\delta;} \\
    \delta\log(\delta^{-1})+\kappa{'}_2(\delta-)(u-\delta), & {u>\delta.}
  \end{array}
\right.$\\
$\kappa_3(u)=\varphi_3(u)=\left\{
  \begin{array}{ll}
    u\log(u^{-1})\log\log(u^{-1}), & {0\leq u\leq\delta;} \\
    \delta\log(\delta^{-1})\log\log(\delta^{-1})+\kappa{'}_3(\delta-)(u-\delta), & {u>\delta,}
  \end{array}
\right.$\\
where $\kappa{'}(\varphi{'})$ denotes the derivative of the function $\kappa(\varphi)$. They are all concave nondecreasing
functions satisfying $\int_{0^+}\frac{1}{\kappa_i(u)+\varphi_i(u)+u}du=\infty(i=1,2,3)$. Furthermore, we observed that the
Lipschitz condition is a special case of our proposed condition.
\end{example}

\begin{assumption}\label{sec2-ass2.11}
$\mathcal{A}$ is a maximal monotone operator with $D(\mathcal{A})=\mathbb{R}^d$.
\end{assumption}

Similar to the proof of the existence and uniqueness of solutions (\cite{J. Gong}, \cite{Ren3}), we can easily obtain the following result under the Assumptions \ref{sec2-ass2.9} and \ref{sec2-ass2.11}. We omit the proof.
\begin{theorem}\label{sec2-th2.12}
If $X_0=\xi$ is a $\mathscr{F}_0$-measurable random variable with $\mathbb{E}|\xi|^2<\infty$ and Assumptions \ref{sec2-ass2.9}, \ref{sec2-ass2.11} hold, then there exists a unique solution $(X,K)$ to Equation \eqref{sec1-eq1.1}. Moreover, this solution satisfies
$\mathbb{E}(\sup_{0\leq{t}\leq{T}}|X_t|^2)\leq C$.
\end{theorem}

\begin{lemma}\label{sec2-lem2.13}(Bihari's inequality~~\cite{Mao3} )
Let $T>0$, and $c>0$. Let $\psi: \mathbb{R}_+\rightarrow\mathbb{R}_+$ be a continuous nondecreasing function such that $\psi(t)>0$ for all $t>0$. Let $u(\cdot)$ be a Borel measurable bounded nonnegative function on [0,T], and let $v(\cdot)$ be a nonnegative integrable function on [0,T]. If
$
u(t)\leq c+\int_0^tv(s)\psi(u(s))ds$  for all $0\leq t\leq T,
$
then
$$
u(t)\leq G^{-1}(G(c)+\int_0^tv(s)ds)
$$
holds for all $t\in[0,T]$ such that
$$
G(c)+\int_0^tv(s)ds\in Dom(G^{-1}),
$$
where
$
G(r)=\int_1^r\frac{ds}{\psi(s)}$  on $r>0,
$
and $G^{-1}$ is the inverse function of G.
\end{lemma}

\section{Stochastic averaging principle}\label{sec3}
In this section, we shall study the averaging principle for multi-valued McKean-Vlasov stochastic differential equations with jumps. Let us consider the standard form of \eqref{sec3-eq3.1}
\begin{equation}\label{sec3-eq3.1}
\begin{split}
X^\epsilon_t=\xi-\epsilon K_t
&+\epsilon\int_0^tb(s,X^\epsilon_s,\mathscr{L}_{X^\epsilon_s})ds+\sqrt\epsilon\int_0^t\sigma(s,X^\epsilon_s,\mathscr{L}_{X^\epsilon_s})dB_s\\
&+\sqrt\epsilon\int_{[0,t]\times U_0}
f(s,X^\epsilon_{s-},\mathscr{L}_{X^\epsilon_s},u)\widetilde{N}(ds,du)
\end{split}
\end{equation}
with the initial value $X^\epsilon_0=\xi$. Here the coefficients $b,\sigma$ and $f$ have the same conditions as in Assumption \ref{sec2-ass2.9} and $\epsilon\in[0,\epsilon_0]$ is a positive small parameter with $\epsilon_0$ is  a fixed number. Thus, under Assumptions \ref{sec2-ass2.9} and \ref{sec2-ass2.11}, \eqref{sec3-eq3.1} has a unique solution $(X^\epsilon_t,K_t),t\in[0,T]$, Moreover, this solution satisfies
$$
\mathbb{E}(\sup_{0\leq{t}\leq{T}}|X^\epsilon_t|^2)\leq C_\epsilon.
$$

Our objective is to show that the solution $(X^\epsilon_t,K_t),t\in[0,T]$ could be approximated in the sense of the mean square convergence by the solution $(Y^\epsilon_t,\bar{K}_t),t\in[0,T]$ of the following averaged equation
\begin{equation}\label{sec3-eq3.2}
\begin{split}
Y^\epsilon_t=\xi-\epsilon\bar{K}_t
&+\epsilon\int_0^t\bar{b}(Y^\epsilon_s,\mathscr{L}_{Y^\epsilon_s})ds+\sqrt\epsilon\int_0^t\bar{\sigma}(Y^\epsilon_s,\mathscr{L}_{Y^\epsilon_s})dB_s\\
&+\sqrt\epsilon\int_{[0,t]\times U_0}
\bar{f}(Y^\epsilon_{s-},\mathscr{L}_{Y^\epsilon_s},u)\widetilde{N}(ds,du),
\end{split}
\end{equation}
where $\bar{b}:\mathbb{R}^d\times\mathcal{M}_2(\mathbb{R}^d)\rightarrow\mathbb{R}^d$, $\bar{\sigma}:\mathbb{R}^d\times\mathcal{M}_2(\mathbb{R}^d)\rightarrow\mathbb{R}^{d\times m}$ and $\bar{f}:\mathbb{R}^d\times\mathcal{M}_2(\mathbb{R}^d)\times\mathbb{U}_0\rightarrow\mathbb{R}^d$ are Borel measurable functions.

\begin{assumption}\label{sec3-ass3.1}(Averaging conditions) For any $x\in\mathbb{R}^d$, $\mu\in\mathcal{M}_2(\mathbb{R}^d)$ and $T_1>0$,
there exist three positive bounded functions $\psi_i:(0,\infty)\rightarrow(0,\infty),i=1,2,3$ with $\lim_{{T_1}\to\infty}\psi_i({T_1})=0$, such that
$$
\frac1{T_1}\int_0^{T_1}|b(s,x,\mu)-\bar{b}(x,\mu)|^2ds
\leq\psi_1(T_1)(1+|x|^2+\|\mu\|_2^2),
$$
$$
\frac1{T_1}\int_0^{T_1}\|\sigma(s,x,\mu)-\bar{\sigma}(x,\mu)\|^2ds
\leq\psi_2(T_1)(1+|x|^2+\|\mu\|_2^2),
$$
$$
\frac1{T_1}\int_0^{T_1}|f(s,x,\mu,u)-\bar{f}(x,\mu,u)|^2ds
\leq\psi_3(T_1)\|u\|_{\mathbb{U}}^2(1+|x|^2+\|\mu\|_2^2).
$$
\end{assumption}

\begin{remark}\label{sec3-rem3.1}
For any $x,y\in\mathbb{R}^d$, $\mu,\nu\in\mathcal{M}_2(\mathbb{R}^d)$ and $T_1>0$, we have
\begin{equation}\label{sec3-eq3.3}
\begin{split}
&|\bar{b}(x,\mu)-\bar{b}(y,\nu)|^2+\|\bar{\sigma}(x,\mu)-\bar{\sigma}(y,\nu)\|^2\\
&\leq3\frac{1}{T_1}\int_0^{T_1}|b(s,x,\mu)-\bar{b}(x,\mu)|^2ds+3\frac{1}{T_1}\int_0^{T_1}|b(s,y,\nu)-\bar{b}(y,\nu)|^2ds\\
&\quad+3\frac{1}{T_1}\int_0^{T_1}|b(s,x,\mu)-b(s,y,\nu)|^2ds+3\frac{1}{T_1}\int_0^{T_1}\|\sigma(s,x,\mu)-\bar{\sigma}(x,\mu)\|^2ds\\
&\quad+3\frac{1}{T_1}\int_0^{T_1}\|\sigma(s,y,\nu)-\bar{\sigma}(y,\nu)\|^2ds+3\frac{1}{T_1}\int_0^{T_1}\|\sigma(s,x,\mu)-\sigma(s,y,\nu)\|^2ds\\
&\leq3(\psi_1(T_1)+\psi_2(T_1))(2+|x|^2+|y|^2+\|\mu\|_2^2+\|\nu\|_2^2)+3L_1(T_1)\kappa(\beta|x-y|^2+\rho^2(\mu,\nu)),
\end{split}
\end{equation}
\begin{align*}
|\bar{b}(x,\mu)|^2+\|\bar{\sigma}(x,\mu)\|^2
&\leq2\frac{1}{T_1}\int_0^{T_1}|b(s,x,\mu)-\bar{b}(x,\mu)|^2ds+2\frac{1}{T_1}\int_0^{T_1}|b(s,x,\mu)|^2ds\\
&\quad+2\frac{1}{T_1}\int_0^{T_1}\|\sigma(s,x,\mu)-\bar{\sigma}(x,\mu)\|^2ds+2\frac{1}{T_1}\int_0^{T_1}\|\sigma(s,x,\mu)\|^2ds\\
&\leq 2[\psi_1(T_1)+\psi_2(T_1)+L_2(T_1)](1+|x|^2+\|\mu\|_2^2).
\end{align*}
Taking $T_1\rightarrow\infty$, because $L_1(t)$ and $L_2(t)$ are bounded, there exists a constant M, such that
$$
|\bar{b}(x,\mu)-\bar{b}(y,\nu)|^2+\|\bar{\sigma}(x,\mu)-\bar{\sigma}(y,\nu)\|^2\leq M\kappa(\beta|x-y|^2+\rho^2(\mu,\nu)),
$$
$$
|\bar{b}(x,\mu)|^2+\|\bar{\sigma}(x,\mu)\|^2\leq M(1+|x|^2+\|\mu\|_2^2).
$$

\noindent{Similarly, for any x, y $\in\mathbb{R}^d$, $\mu,\nu\in\mathcal{M}_2(\mathbb{R}^d)$ and $T_1>0$, we
have}
\begin{equation}\label{sec3-eq3.4}
\begin{split}
&\quad|\bar{f}(x,\mu,u)-\bar{f}(y,\nu,u)|^2\\
&\leq3\frac{1}{T_1}\int_0^{T_1}|f(s,x,\mu,u)-\bar{f}(x,\mu,u)|^2ds
+3\frac{1}{T_1}\int_0^{T_1}|f(s,y,\nu,u)-\bar{f}(y,\nu,u)|^2ds\\
&\quad+3\frac{1}{T_1}\int_0^{T_1}|f(s,x,\mu,u)-f(s,y,\nu,u)|^2ds\\
&\leq3\psi_3(T_1)\|u\|_{\mathbb{U}}^2(2+|x|^2+|y|^2+\|\mu\|_2^2+\|\nu\|_2^2)
+3L_1(T_1)\|u\|_{\mathbb{U}}^2\varphi(\beta|x-y|^2+\rho^2(\mu,\nu)),
\end{split}
\end{equation}
\begin{align*}
|\bar{f}(x,\mu,u)|^2
&\leq2\frac{1}{T_1}\int_0^{T_1}|f(s,x,\mu,u)-\bar{f}(x,\mu,u)|^2ds
+2\frac{1}{T_1}\int_0^{T_1}|f(s,x,\mu,u)|^2ds\\
&\leq2(\psi_3(T_1)+L_2(T_1))\|u\|_{\mathbb{U}}^2(1+|x|^2+\|\mu\|_2^2).
\end{align*}
Taking $T_1\rightarrow\infty$, we have
$$
|\bar{f}(x,\mu,u)-\bar{f}(y,\nu,u)|^2\leq M\|u\|_{\mathbb{U}}^2\varphi(\beta|x-y|^2+\rho^2(\mu,\nu)),
$$
$$
|\bar{f}(x,\mu,u)|^2\leq M\|u\|_{\mathbb{U}}^2(1+|x|^2+\|\mu\|_2^2).
$$
Thus, the coefficients $\bar{b}, \bar{\sigma}, \bar{f}$ satisfy Assumption \ref{sec2-ass2.9}.
Therefore, under Assumption \ref{sec2-ass2.9} and \ref{sec2-ass2.11}, there is a unique solution $(Y^\epsilon_.,\bar{K}_.)$ to the averaged equation ({\ref{sec3-eq3.2}}).
Moreover, this solution satisfies $$\mathbb{E}(\sup_{0\leq{t}\leq{T}}|Y^{\epsilon}_t|^2)\leq{C_{\epsilon}}.$$
\end{remark}

\begin{theorem}\label{sec3-th3.2}
Suppose that $\mathbb{E}|\xi|^2<+\infty$. Then under Assumptions \ref{sec2-ass2.9}, \ref{sec2-ass2.11} and \ref{sec3-ass3.1}, the following averaging principle holds
$$
\lim_{\epsilon\rightarrow0}\mathbb{E}(\sup_{0\leq t\leq T}|X^{\epsilon}_t-Y^{\epsilon}_t|^2)=0.
$$
\end{theorem}
\begin{proof}
It comes from ({\ref{sec3-eq3.1}}) and ({\ref{sec3-eq3.2}}), we have
\begin{align*}
X^{\epsilon}_t-Y^{\epsilon}_t=-\epsilon[K_t-\bar{K}_t]
&+\epsilon\int_0^t[b(s,X^\epsilon_s,\mathscr{L}_{X^\epsilon_s})-\bar{b}(Y^\epsilon_s,\mathscr{L}_{Y^\epsilon_s})]ds\\
&+\sqrt\epsilon\int_0^t[\sigma(s,X^\epsilon_s,\mathscr{L}_{X^\epsilon_s})-\bar{\sigma}(Y^\epsilon_s,\mathscr{L}_{Y^\epsilon_s})]dB_s\\
&+\sqrt\epsilon\int_{[0,t]\times U_0}
[f(s,X^\epsilon_{s-},\mathscr{L}_{X^\epsilon_s},u)-\bar{f}(Y^\epsilon_{s-},\mathscr{L}_{Y^\epsilon_s},u)]\widetilde{N}(ds,du).
\end{align*}
Using It\^{o} formula, we have
\begin{equation}\label{sec3-eq3.5}
\begin{split}
&\quad|X^{\epsilon}_t-Y^{\epsilon}_t|^2\\
&=-2\epsilon\int_0^t\langle X^\epsilon_s-Y^{\epsilon}_s,dK_s-d\bar{K}_s\rangle
+2\epsilon\int_0^t\langle X^\epsilon_s-Y^{\epsilon}_s,
b(s,X^\epsilon_s,\mathscr{L}_{X^\epsilon_s})-\bar{b}(Y^\epsilon_s,\mathscr{L}_{Y^\epsilon_s})\rangle ds\\
&\quad+2\sqrt\epsilon\int_0^t\langle X^\epsilon_s-Y^{\epsilon}_s,
(\sigma(s,X^\epsilon_s,\mathscr{L}_{X^\epsilon_s})-\bar{\sigma}(Y^\epsilon_s,\mathscr{L}_{Y^\epsilon_s}))dB_s\rangle\\
&\quad+\epsilon\int_0^t\|\sigma(s,X^\epsilon_s,\mathscr{L}_{X^\epsilon_s})-\bar{\sigma}(Y^\epsilon_s,\mathscr{L}_{Y^\epsilon_s})\|^2ds\\
&\quad+\epsilon\int_{[0,t]\times U_0}
|f(s,X^\epsilon_{s},\mathscr{L}_{X^\epsilon_s},u)-\bar{f}(Y^\epsilon_{s},\mathscr{L}_{Y^\epsilon_s},u)|^2v(du)ds\\
&\quad+\epsilon\int_{[0,t]\times U_0}
|f(s,X^\epsilon_{s-},\mathscr{L}_{X^\epsilon_s},u)-\bar{f}(Y^\epsilon_{s-},\mathscr{L}_{Y^\epsilon_s},u)|^2\widetilde{N}(ds,du)\\
&\quad+2\sqrt\epsilon\int_{[0,t]\times U_0}
\langle X^\epsilon_{s-}-Y^{\epsilon}_{s-},
f(s,X^\epsilon_{s-},\mathscr{L}_{X^\epsilon_s},u)-\bar{f}(Y^\epsilon_{s-},\mathscr{L}_{Y^\epsilon_s},u)\rangle\widetilde{N}(ds,du).
\end{split}
\end{equation}

\noindent{Then, according to Lemma \ref{sec2-lem2.3}, it is immediate to conclude that}
$$
\int_0^t\langle X^\epsilon_s-Y^{\epsilon}_s,dK_s-d\bar{K}_s\rangle\geq0.
$$

\noindent{Taking the expectation on both sides of \eqref{sec3-eq3.5}, it follows that for any $t\in[0,T]$,}
\begin{equation}\label{sec3-eq3.6}
\begin{split}
&\quad\mathbb{E}(\sup_{0\leq s\leq t}|X^{\epsilon}_s-Y^{\epsilon}_s|^2)\\
&\leq2\epsilon\mathbb{E}\int_0^t|X^{\epsilon}_s-Y^{\epsilon}_s|
|b(s,X^\epsilon_s,\mathscr{L}_{X^\epsilon_s})-\bar{b}(Y^\epsilon_s,\mathscr{L}_{Y^\epsilon_s})|ds\\
&\quad+\epsilon\mathbb{E}\int_0^t
\|\sigma(s,X^\epsilon_s,\mathscr{L}_{X^\epsilon_s})-\bar{\sigma}(Y^\epsilon_s,\mathscr{L}_{Y^\epsilon_s})\|^2ds\\
&\quad+\epsilon\mathbb{E}\int_{[0,t]\times U_0}
|f(s,X^\epsilon_{s},\mathscr{L}_{X^\epsilon_s},u)-\bar{f}(Y^\epsilon_{s},\mathscr{L}_{Y^\epsilon_s},u)|^2v(du)ds\\
&\quad+2\sqrt\epsilon\mathbb{E}\Big(\sup_{0\leq s\leq t}\int_0^s\langle X^\epsilon_r-Y^{\epsilon}_r,
(\sigma(r,X^\epsilon_r,\mathscr{L}_{X^\epsilon_r})-\bar{\sigma}(Y^\epsilon_r,\mathscr{L}_{Y^\epsilon_r}))dB_r\rangle\Big)\\
&\quad+\epsilon\mathbb{E}\Big(\sup_{0\leq s\leq t}\int_{[0,s]\times U_0}
|f(r,X^\epsilon_{r-},\mathscr{L}_{X^\epsilon_r},u)-\bar{f}(Y^\epsilon_{r-},\mathscr{L}_{Y^\epsilon_r},u)|^2\widetilde{N}(dr,du)\Big)\\
&\quad+2\sqrt\epsilon\mathbb{E}\Big(\sup_{0\leq s\leq t}\int_{[0,s]\times U_0}
\langle X^\epsilon_{r-}-Y^{\epsilon}_{r-},
f(r,X^\epsilon_{r-},\mathscr{L}_{X^\epsilon_r},u)-\bar{f}(Y^\epsilon_{r-},\mathscr{L}_{Y^\epsilon_r},u)\rangle\widetilde{N}(dr,du)\Big).
\end{split}
\end{equation}
By the basic inequality $2|a||b|\leq|a|^2+|b|^2$, we can obtain
\begin{equation}\label{sec3-eq3.7}
\begin{split}
&\quad2\epsilon\mathbb{E}\int_0^t|X^{\epsilon}_s-Y^{\epsilon}_s|
|b(s,X^\epsilon_s,\mathscr{L}_{X^\epsilon_s})-\bar{b}(Y^\epsilon_s,\mathscr{L}_{Y^\epsilon_s})|ds\\
&\leq\epsilon\mathbb{E}\int_0^t|X^{\epsilon}_s-Y^{\epsilon}_s|^2ds
+\epsilon\mathbb{E}\int_0^t|b(s,X^\epsilon_s,\mathscr{L}_{X^\epsilon_s})-\bar{b}(Y^\epsilon_s,\mathscr{L}_{Y^\epsilon_s})|^2ds.\\
\end{split}
\end{equation}
The Burkholder-Davis-Gundy's inequality implies that
\begin{equation}\label{sec3-eq3.8}
\begin{split}
&\quad2\sqrt\epsilon\mathbb{E}\Big(\sup_{0\leq s\leq t}\int_0^s\langle X^\epsilon_r-Y^{\epsilon}_r,
(\sigma(r,X^\epsilon_r,\mathscr{L}_{X^\epsilon_r})-\bar{\sigma}(Y^\epsilon_r,\mathscr{L}_{Y^\epsilon_r}))dB_r\rangle\Big)\\
&\leq C\sqrt\epsilon\mathbb{E}\Big[\int_0^t|\langle X^\epsilon_s-Y^{\epsilon}_s,
\sigma(s,X^\epsilon_s,\mathscr{L}_{X^\epsilon_s})-\bar{\sigma}(Y^\epsilon_s,\mathscr{L}_{Y^\epsilon_s})\rangle|^2ds\Big]^{\frac12}\\
&\leq C\sqrt\epsilon\mathbb{E}\Big[\int_0^t|X^\epsilon_s-Y^{\epsilon}_s|^2
\|\sigma(s,X^\epsilon_s,\mathscr{L}_{X^\epsilon_s})-\bar{\sigma}(Y^\epsilon_s,\mathscr{L}_{Y^\epsilon_s})\|^2ds\Big]^{\frac12}\\
&\leq C\sqrt\epsilon\mathbb{E}\Big[\sup_{0\leq s\leq t}|X^\epsilon_s-Y^{\epsilon}_s|^2\int_0^t
\|\sigma(s,X^\epsilon_s,\mathscr{L}_{X^\epsilon_s})-\bar{\sigma}(Y^\epsilon_s,\mathscr{L}_{Y^\epsilon_s})\|^2ds\Big]^{\frac12}\\
&\leq\frac13\mathbb{E}(\sup_{0\leq s\leq t}|X^\epsilon_s-Y^{\epsilon}_s|^2)
+C\epsilon\mathbb{E}\int_0^t\|\sigma(s,X^\epsilon_s,\mathscr{L}_{X^\epsilon_s})-\bar{\sigma}(Y^\epsilon_s,\mathscr{L}_{Y^\epsilon_s})\|^2ds.\\
\end{split}
\end{equation}
By Lemma \ref{sec2-lem2.1}, we have
\begin{equation}\label{sec3-eq3.9}
\begin{split}
&\quad\epsilon\mathbb{E}\Big(\sup_{0\leq s\leq t}\int_{[0,s]\times U_0}
|f(r,X^\epsilon_{r-},\mathscr{L}_{X^\epsilon_r},u)-\bar{f}(Y^\epsilon_{r-},\mathscr{L}_{Y^\epsilon_r},u)|^2\widetilde{N}(dr,du)\Big)\\
&\leq C\epsilon\mathbb{E}\Big\{\sum_{s\in D_{p(s)},s\leq t}
|f(s,X^\epsilon_{s-},\mathscr{L}_{X^\epsilon_{s}},p(s))
-\bar{f}(Y^\epsilon_{s-},\mathscr{L}_{Y^\epsilon_s},p(s))|^4\Big\}^{\frac12}\\
&\leq C\epsilon\mathbb{E}\sum_{s\in D_{p(s)},s\leq t}
|f(s,X^\epsilon_{s-},\mathscr{L}_{X^\epsilon_{s}},p(s))
-\bar{f}(Y^\epsilon_{s-},\mathscr{L}_{Y^\epsilon_s},p(s))|^2\\
&= C\epsilon\mathbb{E}\int_{[0,t]\times U_0}|f(s,X^\epsilon_{s},\mathscr{L}_{X^\epsilon_s},u)-\bar{f}(Y^\epsilon_{s},\mathscr{L}_{Y^\epsilon_s},u)|^2v(du)ds,
\end{split}
\end{equation}
and
\begin{equation}\label{sec3-eq3.10}
\begin{split}
&\quad2\sqrt\epsilon\mathbb{E}\Big(\sup_{0\leq s\leq t}\int_{[0,s]\times U_0}
\langle X^\epsilon_{r-}-Y^{\epsilon}_{r-},
f(r,X^\epsilon_{r-},\mathscr{L}_{X^\epsilon_r},u)-\bar{f}(Y^\epsilon_{r-},\mathscr{L}_{Y^\epsilon_r},u)\rangle\widetilde{N}(dr,du)\Big)\\
&\leq C\sqrt\epsilon\mathbb{E}\Big[\int_{[0,t]\times U_0}
|\langle X^\epsilon_{s-}-Y^{\epsilon}_{s-},
f(s,X^\epsilon_{s-},\mathscr{L}_{X^\epsilon_s},u)-\bar{f}(Y^\epsilon_{s-},\mathscr{L}_{Y^\epsilon_s},u)\rangle|^2 N(ds,du)\Big]^\frac12\\
&\leq C\sqrt\epsilon\mathbb{E}\Big[\sup_{0\leq s\leq t}|X^\epsilon_s-Y^{\epsilon}_s|^2\int_{[0,t]\times U_0}
|f(s,X^\epsilon_{s-},\mathscr{L}_{X^\epsilon_s},u)-\bar{f}(Y^\epsilon_{s-},\mathscr{L}_{Y^\epsilon_s},u)|^2 N(ds,du)\Big]^\frac12\\
&\leq\frac13\mathbb{E}(\sup_{0\leq s\leq t}|X^\epsilon_s-Y^{\epsilon}_s|^2)
+C\epsilon\mathbb{E}\int_{[0,t]\times U_0}
|f(s,X^\epsilon_{s},\mathscr{L}_{X^\epsilon_s},u)-\bar{f}(Y^\epsilon_{s},\mathscr{L}_{Y^\epsilon_s},u)|^2v(du)ds.\\
\end{split}
\end{equation}
Combing with \eqref{sec3-eq3.6}-\eqref{sec3-eq3.10}, we obtain
\begin{equation}\label{sec3-eq3.11}
\begin{split}
&\quad\mathbb{E}(\sup_{0\leq s\leq t}|X^{\epsilon}_s-Y^{\epsilon}_s|^2)\\
&\leq3\epsilon\mathbb{E}\int_0^t|X^{\epsilon}_s-Y^{\epsilon}_s|^2ds
+3\epsilon\mathbb{E}\int_0^t|b(s,X^\epsilon_s,\mathscr{L}_{X^\epsilon_s})-\bar{b}(Y^\epsilon_s,\mathscr{L}_{Y^\epsilon_s})|^2ds\\
&\quad+C\epsilon\mathbb{E}\int_0^t\|\sigma(s,X^\epsilon_s,\mathscr{L}_{X^\epsilon_s})-\bar{\sigma}(Y^\epsilon_s,\mathscr{L}_{Y^\epsilon_s})\|^2ds\\
&\quad+C\epsilon\mathbb{E}\int_{[0,t]\times U_0}
|f(s,X^\epsilon_{s},\mathscr{L}_{X^\epsilon_s},u)-\bar{f}(Y^\epsilon_{s},\mathscr{L}_{Y^\epsilon_s},u)|^2v(du)ds.\\
\end{split}
\end{equation}
By Assumption \ref{sec2-ass2.9} and $\int_{\mathbb{U}_0}\|u\|^2_{\mathbb{U}}v(du)<\infty$, we have
\begin{align*}
&\mathbb{E}(\sup_{0\leq s\leq t}|X^{\epsilon}_s-Y^{\epsilon}_s|^2)\\
&\leq3\epsilon\mathbb{E}\int_0^t|X^{\epsilon}_s-Y^{\epsilon}_s|^2ds
+6\epsilon\mathbb{E}\int_0^t|b(s,X^\epsilon_s,\mathscr{L}_{X^\epsilon_s})-b(s,Y^\epsilon_s,\mathscr{L}_{Y^\epsilon_s})|^2ds\\
&\quad+C\epsilon\mathbb{E}\int_0^t\|\sigma(s,X^\epsilon_s,\mathscr{L}_{X^\epsilon_s})-\sigma(s,Y^\epsilon_s,\mathscr{L}_{Y^\epsilon_s})\|^2ds\\
&\quad+C\epsilon\mathbb{E}\int_{[0,t]\times U_0}
|f(s,X^\epsilon_{s},\mathscr{L}_{X^\epsilon_s},u)-f(s,Y^\epsilon_{s},\mathscr{L}_{Y^\epsilon_s},u)|^2v(du)ds\\
\end{align*}
\begin{equation}\label{sec3-eq3.12}
\begin{split}
&\quad+6\epsilon\mathbb{E}\int_0^t|b(s,Y^\epsilon_s,\mathscr{L}_{Y^\epsilon_s})-\bar{b}(Y^\epsilon_s,\mathscr{L}_{Y^\epsilon_s})|^2ds
+C\epsilon\mathbb{E}\int_0^t\|\sigma(s,Y^\epsilon_s,\mathscr{L}_{Y^\epsilon_s})-\bar{\sigma}(Y^\epsilon_s,\mathscr{L}_{Y^\epsilon_s})\|^2ds\\
&\quad+C\epsilon\mathbb{E}\int_{[0,t]\times U_0}
|f(s,Y^\epsilon_{s},\mathscr{L}_{Y^\epsilon_s},u)-\bar{f}(Y^\epsilon_{s},\mathscr{L}_{Y^\epsilon_s},u)|^2v(du)ds\\
&\leq3\epsilon\int_0^t\mathbb{E}|X^{\epsilon}_s-Y^{\epsilon}_s|^2ds
+C\epsilon\int_0^tL_1(s)\Big[\kappa\Big((\beta+1)\mathbb{E}|X^{\epsilon}_s-Y^{\epsilon}_s|^2\Big)+\varphi\Big((\beta+1)\mathbb{E}|X^{\epsilon}_s-Y^{\epsilon}_s|^2\Big)\Big]ds\\
&\quad+6\epsilon\mathbb{E}\int_0^t|b(s,Y^\epsilon_s,\mathscr{L}_{Y^\epsilon_s})-\bar{b}(Y^\epsilon_s,\mathscr{L}_{Y^\epsilon_s})|^2ds
+C\epsilon\mathbb{E}\int_0^t\|\sigma(s,Y^\epsilon_s,\mathscr{L}_{Y^\epsilon_s})-\bar{\sigma}(Y^\epsilon_s,\mathscr{L}_{Y^\epsilon_s})\|^2ds\\
&\quad+C\epsilon\mathbb{E}\int_{[0,t]\times U_0}
|f(s,Y^\epsilon_{s},\mathscr{L}_{Y^\epsilon_s},u)-\bar{f}(Y^\epsilon_{s},\mathscr{L}_{Y^\epsilon_s},u)|^2v(du)ds.\\
\end{split}
\end{equation}
By Assumption \ref{sec2-ass2.9}, Assumption \ref{sec3-ass3.1} and Remark \ref{sec3-rem3.1}, we obtain
\begin{equation}\label{sec3-eq3.13}
\begin{split}
&\quad6\epsilon\mathbb{E}\int_0^t|b(s,Y^\epsilon_s,\mathscr{L}_{Y^\epsilon_s})-\bar{b}(Y^\epsilon_s,\mathscr{L}_{Y^\epsilon_s})|^2ds\\
&\leq18\epsilon\mathbb{E}\int_0^t|b(s,Y^\epsilon_s,\mathscr{L}_{Y^\epsilon_s})-b(s,\xi,\mathscr{L}_{\xi})|^2ds
+18\epsilon\mathbb{E}\int_0^t|b(s,\xi,\mathscr{L}_{\xi})-\bar{b}(\xi,\mathscr{L}_{\xi})|^2ds\\
&\quad+18\epsilon\mathbb{E}\int_0^t|\bar{b}(\xi,\mathscr{L}_{\xi})-\bar{b}(Y^\epsilon_s,\mathscr{L}_{Y^\epsilon_s})|^2ds\\
&\leq18\epsilon\mathbb{E}\int_0^t(L_1(s)+M)\kappa\Big(\beta|Y^\epsilon_s-\xi|^2+\rho^2(\mathscr{L}_{Y^\epsilon_s},\mathscr{L}_{\xi})\Big)ds
+18\epsilon\psi_1(t)\mathbb{E}(1+|\xi|^2+\|\mathscr{L}_{\xi}\|_2^2)t\\
&\leq18\epsilon\int_0^t(L_1(s)+M)\kappa\Big((\beta+1)\mathbb{E}|Y^\epsilon_s-\xi|^2\Big)ds
+18\epsilon\psi_1(t)(1+2\mathbb{E}|\xi|^2)t.\\
\end{split}
\end{equation}
By the similar deduction to that in \eqref{sec3-eq3.13}, we obtain that
\begin{equation}\label{sec3-eq3.14}
\begin{split}
&\quad C\epsilon\mathbb{E}
\int_0^t\|\sigma(s,Y^\epsilon_s,\mathscr{L}_{Y^\epsilon_s})-\bar{\sigma}(Y^\epsilon_s,\mathscr{L}_{Y^\epsilon_s})\|^2ds\\
&\leq C\epsilon\int_0^t(L_1(s)+M)\kappa\Big((\beta+1)\mathbb{E}|Y^\epsilon_s-\xi|^2\Big)ds
+C\epsilon\psi_2(t)(1+2\mathbb{E}|\xi|^2)t,\\
\end{split}
\end{equation}
and
\begin{equation}\label{sec3-eq3.15}
\begin{split}
&\quad C\epsilon\mathbb{E}\int_{[0,t]\times U_0}
|f(s,Y^\epsilon_{s},\mathscr{L}_{Y^\epsilon_s},u)-\bar{f}(Y^\epsilon_{s},\mathscr{L}_{Y^\epsilon_s},u)|^2v(du)ds\\
&\leq C\epsilon\int_0^t(L_1(s)+M)\varphi\Big((\beta+1)\mathbb{E}|Y^\epsilon_s-\xi|^2\Big)ds
+C\epsilon\psi_3(t)(1+2\mathbb{E}|\xi|^2)t.\\
\end{split}
\end{equation}
Using \eqref{sec3-eq3.12}-\eqref{sec3-eq3.15}, we can get that
\begin{align*}
&\mathbb{E}(\sup_{0\leq s\leq t}|X^{\epsilon}_s-Y^{\epsilon}_s|^2)\\
&\leq3\epsilon\int_0^t\mathbb{E}|X^{\epsilon}_s-Y^{\epsilon}_s|^2ds
+C\epsilon\int_0^tL_1(s)\Big[\kappa\Big((\beta+1)\mathbb{E}|X^{\epsilon}_s-Y^{\epsilon}_s|^2\Big)
+\varphi\Big((\beta+1)\mathbb{E}|X^{\epsilon}_s-Y^{\epsilon}_s|^2\Big)\Big]ds\\
\end{align*}
\begin{align*}
&+C\epsilon\int_0^t(L_1(s)+M)\Big[\kappa\Big((\beta+1)\mathbb{E}|Y^\epsilon_s-\xi|^2\Big)
+\varphi\Big((\beta+1)\mathbb{E}|Y^\epsilon_s-\xi|^2\Big)\Big]ds\\
&+C(\psi_1(t)+\psi_2(t)+\psi_3(t))(1+2\mathbb{E}|\xi|^2)\epsilon t.\\
\end{align*}
Let $\gamma(x)=x+\kappa(x)+\varphi(x)$, by Remark \ref{sec3-rem3.1} and the boundness of $\psi_i(t),i=1,2,3$, we can obtain
\begin{align*}
&\quad\mathbb{E}(\sup_{0\leq s\leq t}|X^{\epsilon}_s-Y^{\epsilon}_s|^2)\\
&\leq C\epsilon\int_0^t(L_1(s)+1)\gamma\Big((\beta+1)\mathbb{E}|X^{\epsilon}_s-Y^{\epsilon}_s|^2\Big)ds
+C\epsilon\int_0^t(L_1(s)+M)\gamma\Big((\beta+1)\mathbb{E}|Y^\epsilon_s-\xi|^2\Big)ds\\
&\quad+C(\psi_1(t)+\psi_2(t)+\psi_3(t))(1+2\mathbb{E}|\xi|^2)\epsilon t\\
&\leq C\epsilon\int_0^t(L_1(s)+1)\gamma
\Big((\beta+1)\mathbb{E}(\sup_{0\leq r\leq s}|X^{\epsilon}_r-Y^{\epsilon}_r|^2)\Big)ds
+C(\psi_1(t)+\psi_2(t)+\psi_3(t))(1+2\mathbb{E}|\xi|^2)\epsilon t\\
&\quad+C\epsilon\int_0^t(L_1(s)+M)\gamma
\Big(2(\beta+1)\mathbb{E}(\sup_{0\leq r\leq s}|Y^\epsilon_r|^2)+2(\beta+1)\mathbb{E}|\xi|^2\Big)ds\\
&\leq C\epsilon\int_0^t(L_1(s)+1)\gamma
\Big((\beta+1)\mathbb{E}(\sup_{0\leq r\leq s}|X^{\epsilon}_r-Y^{\epsilon}_r|^2)\Big)ds+C\epsilon t.\\
\end{align*}
Obviously, $\gamma(x)$ is nondecreasing function on $\mathbb{R}_+$ and $\gamma(0)=0$.
Setting $G(t)=\int_1^t\frac{ds}{\gamma(s)}$, it follows from Lemma \ref{sec2-lem2.13} that
$$
\mathbb{E}\Big(\sup_{0\leq s\leq t}|X^{\epsilon}_s-Y^{\epsilon}_s|^2\Big)\leq \frac1{\beta+1}G^{-1}\Big(G(C\epsilon t)+C\epsilon(L_1(T)+1)T\Big).
$$
Noting that $C\epsilon t\rightarrow 0$ as $\epsilon\rightarrow0$. Recalling the condition $\int_{0+}\frac{ds}{\gamma(s)}=\infty$, we can conclude that
$$
G(C\epsilon t)+C\epsilon(L_1(T)+1)T\rightarrow-\infty, \quad \epsilon\rightarrow0.
$$
On the other hand, because $G$ is a strictly increasing function, then we obtain that $G$ has an inverse function which is strictly increasing, and $G^{-1}(-\infty)=0$. That is
$$
G^{-1}\Big(G(C\epsilon t)+C\epsilon(L_1(T)+1)T\Big)\rightarrow0,\quad \epsilon\rightarrow0.
$$
Consequently, we have
$$\lim_{\epsilon\rightarrow0}\mathbb{E}(\sup_{0\leq s\leq t}|X^{\epsilon}_s-Y^{\epsilon}_s|^2)=0.$$
Therefore we complete the proof.
\end{proof}

\begin{remark}\label{sec3-rem3.2}
By the Chebyshev-Markov inequality and Theorem \ref{sec3-th3.2}, for any given number $\delta>0$, we have
$$\lim_{\epsilon\to0}\mathbb{P}(\sup_{0\leq{t}\leq{T}}|X^{\epsilon}_t-Y^{\epsilon}_t|>\delta)
\leq\frac1{\delta^2}\lim_{\epsilon\to0}\mathbb{E}(\sup_{0\leq{t}\leq{T}}|X^{\epsilon}_t-Y^{\epsilon}_t|^2)=0.$$
This implies the convergence in probability of the solutions $X^{\epsilon}_t$ to the averaged solution $Y^{\epsilon}_t$.
\end{remark}

\section{The generalized It\^{o} formula}\label{sec4}
In this subsection, we state and prove the It\^{o} formula.
\begin{theorem}\label{sec4-th4.1}
Suppose that Assumptions \ref{sec2-ass2.9} and \ref{sec2-ass2.11} hold. Then, if $h$ belongs to $C_b^{1,2,2}([0,T]\times\mathbb{R}^d\times\mathcal{M}_2(\mathbb{R}^d))$ and all the derivatives of the function $h$ in $(t,x,\mu)$ are uniformly continuous, it holds that for $t\geq 0$,
\begin{equation}\label{sec4-eq4.1}
\begin{split}
dh(t,X_t,\mathscr{L}_{X_t})
&=\partial_th(t,x,\mu)|_{x=X_t,\mu=\mathscr{L}_{X_t}}dt
+\mathbb{L}_{b,\sigma,f}h(t,X_t,\mathscr{L}_{X_t})dt
-\langle\partial_xh(t,X_t,\mathscr{L}_{X_t}),dK_t\rangle\\
&\quad+\langle\partial_xh(t,X_t,\mathscr{L}_{X_t}),\sigma(t,X_t,\mathscr{L}_{X_t})dB_t\rangle
-\mathbb{E}\langle\partial_{\mu }h(s,y,\mathscr{L}_{X_t})(X_t),dK_t\rangle|_{s=t,y=X_t}\\
&\quad+\int_{\mathbb{U}_0}[h(t,X_{t-}+f(t,X_{t-},\mathscr{L}_{X_t},u),\mathscr{L}_{X_t})-h(t,X_{t-},\mathscr{L}_{X_t})]\widetilde{N}(dt,du),
\end{split}
\end{equation}
where
\begin{align*}
&\mathbb{L}_{b,\sigma,f}h(t.x,\mu)\\
&:=\langle\partial_xh,b\rangle(t,x,\mu)
+\frac12tr((\sigma\sigma^*)\partial_x^2h)(t,x,\mu)+\int_{\mathbb{R}^d}\langle b(t,y,\mu),(\partial_{\mu}h)(t,x,\mu)(y)\rangle\mu(dy)\\
&\quad+\frac12\int_{\mathbb{R}^d}tr((\sigma\sigma^*)(t,y,\mu)\partial_y\partial_{\mu}h(t,x,\mu)(y))\mu(dy)\\
&\quad+\int_{\mathbb{U}_0}\int_0^1\int_{\mathbb{R}^d}
\Big\langle\partial_{\mu}h(t,x,\mu)(y+\eta f(t,y,\mu,u))-\partial_{\mu}h(t,x,\mu)(y),f(t,y,\mu,u)\Big\rangle\mu(dy)d\eta v(du)\\
&\quad+\int_{\mathbb{U}_0}[h(t,x+f(t,x,\mu,u),\mu)-h(t,x,\mu)-\langle f(t,x,\mu,u),\partial_xh(t,x,\mu)\rangle]v(du).
\end{align*}
\end{theorem}
\begin{proof}
Let $\mu_t:=\mathscr{L}_{X_t}$, and $h(t,X_t,\mathscr{L}_{X_t})=h(t,X_t,\mu_t)$. Define $\bar{h}(t,x):=h(t,x,\mu_t)$, and then $\bar{h}(t,X_t):=h(t,X_t,\mu_t)$. Moreover, according to $h\in C_b^{1,2,2}([0,T]\times\mathbb{R}^d\times\mathcal{M}_2(\mathbb{R}^d))$, it is immediate to conclude that $\bar{h}$ is $C^2$ in $x$. Our objective is to study the differentiability of $\bar{h}$ in $t$. Observe that the differentiability of $\bar{h}$ in $t$ comes from two parts-$h(t,x,\mu)$ in $t$ for fixed $x,\mu$ and $h(s,x,\mu_t)$ in $t$ for fixed $s,x$. Therefore, in order to apply the classical It\^{o} formula to $\bar{h}(t,X_t)$, we just need to investigate the second part.\\
{\bf Step one.} Suppose that $b,\sigma$ are bounded and  $|f(t,x,\mu,u)|\leq C\|u\|_{\mathbb{U}}$. We consider the differentiability of $h(s,x,\mu_t)$ in $t$ for fixed $s$, $x$.

We follow the method in Chassagneux,   Crisan and   Delarue \cite{J.-F.} to deal with it. For fixed $s$, $x$, we take $H(\mu_t):=h(s,x,\mu_t)$. For any positive integer $N$, let
\begin{equation}\label{sec4-eq4.2}
\begin{split}
x^1,x^2,\cdots,x^N\in\mathbb{R}^d,H^N(x^1,x^2,\cdots,x^N):=H(\frac1N\sum_{l=1}^N\delta_{x^l}),
\end{split}
\end{equation}
and then $H^N(x^1,x^2,\cdots,x^N)$ is a function on $\mathbb{R}^{d\times N}$. In addition, by (\cite{J.-F.}, Proposition 3.1), it holds that $H^N$ is $C^2$ on $\mathbb{R}^{d\times N}$ and
$$
\partial_{x^i}H^N(x^1,x^2,\cdots,x^N)=\frac1N\partial_\mu H\Big(\frac1N\sum_{l=1}^N\delta_{x^l}\Big)(x^i),
$$
\begin{equation}\label{sec4-eq4.3}
\begin{split}
\partial_{x^ix^j}^2H^N(x^1,x^2,\cdots,x^N)=&\frac1N\partial_y\partial_\mu H\Big(\frac1N\sum_{l=1}^N\delta_{x^l}\Big)(x^i)\delta_{i,j}+\frac1{N^2}\partial_\mu^2H\Big(\frac1N\sum_{l=1}^N\delta_{x^l}\Big)(x^i,x^j),
\end{split}
\end{equation}
where $i,j=1,2,\cdots,N$ and $\delta_{i,j}=1,i=j$, $\delta_{i,j}=0,i\neq j$. Moreover, we take $N$ independent copies $X_t^l$, $l=1,2,\cdots,N$ of $X_t$. That is
\begin{align*}
dX_t^l=&-dK_t^l+b(t,X_t^l,\mathscr{L}_{X_t^l})dt+\sigma(t,X_t^l,\mathscr{L}_{X_t^l})dB_t^l\\
&+\int_{\mathbb{U}_0}f(t,X_{t-}^l,\mathscr{L}_{X_t^l},u)\widetilde{N}^l(dt,du),\quad l=1,2,\cdots,N,
\end{align*}
where $B^l,N^l,l=1,2,\cdots,N$ are mutually independent and have the same distributions to that of $B,N$, respectively. Applying the It\^{o} formula to $H^N(X^1_t,X^2_t,\cdots,X^N_t)$ and taking the expectation on both sides, we derive that for $0\leq t<t+v\leq T$
\begin{align*}
&\quad\mathbb{E}H^N(X^1_{t+v},X^2_{t+v},\cdots,X^N_{t+v})\\
&=\mathbb{E}H^N(X^1_{t},X^2_{t},\cdots,X^N_{t})-\sum_{i=1}^N\mathbb{E}\int_t^{t+v}
\Big\langle\partial_{x^i}H^N(X^1_{s},X^2_{s},\cdots,X^N_{s}),dK_s^i\Big\rangle\\
&\quad+\sum_{i=1}^N\int_t^{t+v}\mathbb{E}\partial_{x^i}H^N(X^1_{s},X^2_{s},\cdots,X^N_{s})b(s,X_s^i,\mathscr{L}_{X_s^i})ds\\
&\quad+\frac12\sum_{i=1}^N\int_t^{t+v}\mathbb{E}
tr\Big((\sigma\sigma^*)(s,X_s^i,\mathscr{L}_{X_s^i})\partial_{x^ix^i}^2H^N(X^1_{s},X^2_{s},\cdots,X^N_{s})\Big)ds\\
&\quad+\int_{[t,t+v]\times\mathbb{U}_0}\mathbb{E}\Big[H^N\Big(X^1_{s}+f(s,X_s^1,\mathscr{L}_{X_s^1},u),X^2_{s},\cdots,X^N_{s}\Big)
-H^N(X^1_{s},X^2_{s},\cdots,X^N_{s})\\
&\qquad-\partial_{x^1}H^N(X^1_{s},X^2_{s},\cdots,X^N_{s})f(s,X_s^1,\mathscr{L}_{X_s^1},u)\Big]v(du)ds+\cdots\\
&\quad+\int_{[t,t+v]\times\mathbb{U}_0}\mathbb{E}\Big[H^N\Big(X^1_{s},X^2_{s},\cdots,X^N_{s}+f(s,X_s^N,\mathscr{L}_{X_s^N},u)\Big)
-H^N(X^1_{s},X^2_{s},\cdots,X^N_{s})\\
&\qquad-\partial_{x^N}H^N(X^1_{s},X^2_{s},\cdots,X^N_{s})f(s,X_s^N,\mathscr{L}_{X_s^N},u)\Big]v(du)ds\\
&=\mathbb{E}H^N(X^1_{t},X^2_{t},\cdots,X^N_{t})-N\mathbb{E}\int_t^{t+v}
\Big\langle\partial_{x^1}H^N(X^1_{s},X^2_{s},\cdots,X^N_{s}),dK_s^1\Big\rangle\\
&\quad+N\int_t^{t+v}\mathbb{E}\partial_{x^1}H^N(X^1_{s},X^2_{s},\cdots,X^N_{s})b(s,X_s^1,\mathscr{L}_{X_s^1})ds\\
&\quad+\frac{N}{2}\int_t^{t+v}\mathbb{E}
tr\Big((\sigma\sigma^*)(s,X_s^1,\mathscr{L}_{X_s^1})\partial_{x^1x^1}^2H^N(X^1_{s},X^2_{s},\cdots,X^N_{s})\Big)ds\\
&\quad+N\int_t^{t+v}\int_{\mathbb{U}_0}\int_0^1\mathbb{E}\Big[\Big(\partial_{x^1}H^N(X^1_{s}+\eta f(s,X_s^1,\mathscr{L}_{X_s^1},u),X^2_{s},\cdots,X^N_{s})\\
&\qquad\qquad\qquad\qquad\qquad-\partial_{x^1}H^N(X^1_{s},X^2_{s},\cdots,X^N_{s})\Big)f(s,X_s^1,\mathscr{L}_{X_s^1},u)\Big]d\eta v(du)ds,
\end{align*}
where the property of the same distributions for $X_t^l$, $l=1,2,\cdots,N$ is used in the second equality. Inserting \eqref{sec4-eq4.2} and \eqref{sec4-eq4.3} in the above equality, we obtain

\begin{align*}
&\quad\mathbb{E}H\Big(\frac1N\sum_{l=1}^N\delta_{X^l_{t+v}}\Big)\\
&=\mathbb{E}H\Big(\frac1N\sum_{l=1}^N\delta_{X^l_{t}}\Big)
-\mathbb{E}\int_t^{t+v}\Big\langle\partial_{\mu }
H\Big(\frac1N\sum_{l=1}^N\delta_{X^l_{s}}\Big)(X^1_{s}),dK_s^1\Big\rangle\\
&\quad+\int_t^{t+v}\mathbb{E}\partial_{\mu }
H\Big(\frac1N\sum_{l=1}^N\delta_{X^l_{s}}\Big)(X^1_{s})\times b(s,X^1_{s},\mathscr{L}_{X^1_{s}})ds\\
&\quad+\frac12\int_t^{t+v}\mathbb{E}tr\Big((\sigma\sigma^*)(s,X_s^1,\mathscr{L}_{X_s^1})
\partial_y\partial_{\mu } H\Big(\frac1N\sum_{l=1}^N\delta_{X^l_{s}}\Big)(X^1_{s})\Big)ds\\
&\quad+\frac1{2N}\int_t^{t+v}\mathbb{E}tr\Big((\sigma\sigma^*)(s,X_s^1,\mathscr{L}_{X_s^1})
\partial_{\mu }^2 H\Big(\frac1N\sum_{l=1}^N\delta_{X^l_{s}}\Big)(X^1_{s},X^1_{s})\Big)ds\\
&\quad+\int_t^{t+v}\int_{\mathbb{U}_0}\int_0^1\mathbb{E}\Big[\Big(\partial_{\mu}H\Big(\frac1N\delta_{X^1_{s}+\eta f(s,X_s^1,\mathscr{L}_{X_s^1},u)}+\frac1N\sum_{l=2}^N\delta_{X^l_{s}}\Big)\\
&\qquad\qquad\circ(X^1_{s}+\eta f(s,X_s^1,\mathscr{L}_{X_s^1},u))\\
&\qquad-\partial_{\mu}H\Big(\frac1N\sum_{l=1}^N
\delta_{X^l_{s}}\Big)(X^1_{s})\Big)f(s,X_s^1,\mathscr{L}_{X_s^1},u)\Big]d\eta v(du)ds.\\
\end{align*}
Next, we take the limit on both sides of the above equality. By (\cite{J. Horowitz}, Section 5), it holds that

$$
\lim_{N\rightarrow\infty}\mathbb{E}\Big[\sup_{0\leq t\leq T}\rho^2\Big(\frac1N\sum_{l=1}^N
\delta_{X^l_{t}},\mu_t\Big)\Big]=0,
$$
Letting $N\rightarrow\infty$, by continuity and boundedness of $H$, $\partial_{\mu}H$, $\partial_y\partial_{\mu}H$, and boundedness of $\partial_{\mu}^2H$, $b$, $\sigma$, it follows from the dominated convergence theorem that
\begin{align*}
H(\mu_{t+v})=&H(\mu_{t})-\mathbb{E}\int_t^{t+v}\langle\partial_{\mu }
H(\mu_{s})(X^1_{s}),dK_s^1\rangle\\
&+\int_t^{t+v}\mathbb{E}\partial_{\mu }
H(\mu_{s})(X^1_{s})b(s,X^1_{s},\mathscr{L}_{X^1_{s}})ds\\
&+\frac12\int_t^{t+v}\mathbb{E}tr\Big((\sigma\sigma^*)(s,X_s^1,\mathscr{L}_{X_s^1})
\partial_y\partial_{\mu } H(\mu_{s})(X^1_{s})\Big)ds\\
&+\int_t^{t+v}\int_{\mathbb{U}_0}\int_0^1\mathbb{E}\Big[\Big(\partial_{\mu}H(\mu_{s})
(X^1_{s}+\eta f(s,X_s^1,\mathscr{L}_{X_s^1},u))\\
&\qquad\qquad\qquad-\partial_{\mu}H(\mu_{s})(X^1_{s})\Big)f(s,X_s^1,\mathscr{L}_{X_s^1},u)\Big]d\eta v(du)ds.\\
\end{align*}
\noindent{Consequently, we arrive at}

\begin{equation}\label{sec4-eq4.4}
\begin{split}
\partial_t&H(\mu_{t})=
-\mathbb{E}\langle\partial_{\mu }H(\mu_{t})(X_t),dK_t\rangle+\int_{\mathbb{R}^d}\langle b(t,y,\mu_{t}),\partial_{\mu}H(\mu_{t})(y)\rangle\mu_{t}(dy)\\
&+\frac12\int_{\mathbb{R}^d}tr\Big((\sigma\sigma^*)(t,y,\mu_{t})\partial_y\partial_{\mu}H(\mu_{t})(y)\Big)\mu_{t}(dy)\\
&+\int_{\mathbb{U}_0}\int_0^1\int_{\mathbb{R}^d}
\Big[\Big(\partial_{\mu}H(\mu_{t})(y+\eta f(t,y,\mu_{t},u))-\partial_{\mu}H(\mu_{t})(y)\Big)f(t,y,\mu_{t},u)\Big]\mu_{t}(dy)d\eta v(du).
\end{split}
\end{equation}

\noindent{{\bf Step two.} Suppose that Assumption \ref{sec2-ass2.9} hold. We focus on dealing with the differentiability of $h(s,x,\mu_t)$ in $t$.}

\noindent{To begin with, we choose a smooth function $\chi_n:\mathbb{R}^d\rightarrow\mathbb{R}^d$ satisfying $\chi_n(x)=x,|x|\leq n$ and $\chi_n(x)=0,|x|>2n$ such that for $x\in\mathbb{R}^d$}
\begin{equation}\label{sec4-eq4.5}
\begin{split}
|\chi_n(x)|\leq C_1,\quad \|\partial\chi_n(x)\|\leq C_1,
\end{split}
\end{equation}
where the positive constant $C_1$ is independent of $n$. Set
$$
b^{(n)}(t,x,\mu):=b(t,\chi_n(x),\mu),$$
$$ \sigma^{(n)}(t,x,\mu):=\sigma(t,\chi_n(x),\mu),
$$
$$
f^{(n)}(t,x,\mu,u):=f(t,\chi_n(x),\mu,u),
$$
Letting $n\rightarrow\infty$, we can get
$$
b^{(n)}(t,x,\mu)\rightarrow b(t,x,\mu),\quad \sigma^{(n)}(t,x,\mu)\rightarrow\sigma(t,x,\mu),\quad
f^{(n)}(t,x,\mu,u)\rightarrow f(t,x,\mu,u),
$$
as .

\noindent{Besides, by Assumption \ref{sec2-ass2.9}, we deduce that $b^{(n)},\sigma^{(n)}$ are bounded,} $$|f^{(n)}(t,x,\mu,u)|\leq C\|u\|_{\mathbb{U}}~\hbox{for}~ (t,x,\mu,u)\in[0,T]\times\mathbb{R}^d\times\mathcal{M}_2(\mathbb{R}^d)\times\mathbb{U}_0,$$ and $b^{(n)},\sigma^{(n)},f^{(n)}$ satisfy Assumption \ref{sec2-ass2.9}. Then according to \ref{sec2-th2.12}, the following equation:
\begin{align*}
dX_t^{(n)}\in&-\mathcal{A}(X_t^{(n)})dt+b^{(n)}(t,X_t^{(n)},\mathscr{L}_{X_t^{(n)}})dt
+\sigma^{(n)}(t,X_t^{(n)},\mathscr{L}_{X_t^{(n)}})dB_t\\
&+\int_{U_0}f^{(n)}(t,X_{t-}^{(n)},\mathscr{L}_{X_t^{(n)}},u)\widetilde{N}(dt,du),\quad t\in[0,T]
\end{align*}
has a unique solution $(X_\cdot^{(n)},K_\cdot^{(n)})\in\mathscr{A}$, $\mu_t^{(n)}:=\mathscr{L}_{X_t^{(n)}}$ for $X_0^{(n)}=X_0=\xi$. Thus taking account of {\bf Step one}, it holds that for $0\leq t<t+v\leq T$
\begin{equation}\label{sec4-eq4.6}
\begin{split}
H(\mu_{t+v}^{(n)})-H(\mu_{t}^{(n)})=
&-\int_t^{t+v}\mathbb{E}\langle\partial_{\mu }H(\mu_{r}^{(n)})(X_r^n),dK_r^{(n)}\rangle\\
&+\int_t^{t+v}\int_{\mathbb{R}^d}\Big\langle b^{(n)}(r,y,\mu_{r}^{(n)}),\partial_{\mu}H(\mu_{r}^{(n)})(y)\Big\rangle\mu_{r}^{(n)}(dy)dr\\
&+\frac12\int_t^{t+v}\int_{\mathbb{R}^d}tr\Big((\sigma^{(n)}\sigma^{(n)*})(r,y,\mu_{r}^{(n)})
\partial_y\partial_{\mu}H(\mu_{r}^{(n)})(y)\Big)\mu_{r}^{(n)}(dy)dr\\
&+\int_t^{t+v}\int_{\mathbb{U}_0}\int_0^1\int_{\mathbb{R}^d}
\Big[\Big(\partial_{\mu}H(\mu_{r}^{(n)})(y+\eta f^{(n)}(r,y,\mu_{r}^{(n)},u))\\
&\qquad\qquad\qquad-\partial_{\mu}H(\mu_{r}^{(n)})(y)\Big)f^{(n)}(r,y,\mu_{r}^{(n)},u)\Big]\mu_{r}^{(n)}(dy)d\eta v(du)dr.
\end{split}
\end{equation}

Next, we observe the limit of $\mu_{t}^{(n)}$ as $n\rightarrow\infty$ for any $t\in[0,T]$. Taking the expectation on two sides, the It\^{o} formula implies
\begin{align*}
&\quad\mathbb{E}|X_t^{(n)}-X_t|^2\\
&=-2\mathbb{E}\int_0^t\langle X_r^{(n)}-X_r,dK_r^{(n)}-dK_r\rangle\\
&\quad+2\mathbb{E}\int_0^t\langle b^{(n)}(r,X_r^{(n)},\mu_{r}^{(n)})-b(r,X_r,\mu_{r}),X_r^{(n)}-X_r\rangle dr\\
&\quad+\mathbb{E}\int_0^t\|\sigma^{(n)}(r,X_r^{(n)},\mu_{r}^{(n)})-\sigma(r,X_r,\mu_{r})\|^2dr\\
&\quad+\mathbb{E}\int_0^t\int_{\mathbb{U}_0}
|f^{(n)}(r,X_r^{(n)},\mu_{r}^{(n)},u)-f(r,X_r,\mu_{r},u)|^2v(du)dr\\
&\leq\mathbb{E}\int_0^t|X_r^{(n)}-X_r|^2dr
+\mathbb{E}\int_0^t|b^{(n)}(r,X_r^{(n)},\mu_{r}^{(n)})-b(r,X_r,\mu_{r})|^2dr\\
&\quad+\mathbb{E}\int_0^t\|\sigma^{(n)}(r,X_r^{(n)},\mu_{r}^{(n)})-\sigma(r,X_r,\mu_{r})\|^2dr\\
&\quad+\mathbb{E}\int_0^t\int_{\mathbb{U}_0}
|f^{(n)}(r,X_r^{(n)},\mu_{r}^{(n)},u)-f(r,X_r,\mu_{r},u)|^2v(du)dr.\\
\end{align*}
By Assumption \ref{sec2-ass2.9}, we have
\begin{align*}
&\quad\mathbb{E}|X_t^{(n)}-X_t|^2\\
&\leq\mathbb{E}\int_0^t|X_r^{(n)}-X_r|^2dr
+\mathbb{E}\int_0^tL_1(r)\kappa\Big(\beta|\chi_n(X_r^{(n)})-X_r|^2+\rho^2(\mu_r^{(n)},\mu_r)\Big)dr\\
&\quad+C\mathbb{E}\int_0^tL_1(r)\varphi\Big(\beta|\chi_n(X_r^{(n)})-X_r|^2+\rho^2(\mu_r^{(n)},\mu_r)\Big)dr\\
&\leq\int_0^t\mathbb{E}|X_r^{(n)}-X_r|^2dr
+\int_0^tL_1(r)\kappa\Big(\beta\mathbb{E}|\chi_n(X_r^{(n)})-X_r|^2+\mathbb{E}|X_r^{(n)}-X_r|^2\Big)dr\\
&\quad+C\int_0^tL_1(r)\varphi\Big(\beta\mathbb{E}|\chi_n(X_r^{(n)})-X_r|^2+\mathbb{E}|X_r^{(n)}-X_r|^2\Big)dr.\\
\end{align*}
Letting $\gamma(x)=\kappa(x)+\varphi(x)+x$, we have
$$
\mathbb{E}|X_t^{(n)}-X_t|^2
\leq C\int_0^t(L_1(r)+1)\gamma\Big(\beta\mathbb{E}|\chi_n(X_r^{(n)})-X_r|^2+\mathbb{E}|X_r^{(n)}-X_r|^2\Big)dr.\\
$$
By \eqref{sec4-eq4.5}, we have
\begin{align*}
\mathbb{E}|\chi_n(X_r^{(n)})-X_r|^2
&\leq2\mathbb{E}|\chi_n(X_r^{(n)})-\chi_n(X_r)|^2+2\mathbb{E}|\chi_n(X_r)-X_r|^2\\
&\leq 2C_1\mathbb{E}|X_r^{(n)}-X_r|^2+2\mathbb{E}|\chi_n(X_r)-X_r|^2,
\end{align*}
and furthermore
\begin{align*}
\mathbb{E}|X_t^{(n)}-X_t|^2\leq C\int_0^t(L_1(r)+1)\gamma\Big((2\beta C_1+1)\mathbb{E}|X_r^{(n)}-X_r|^2
+2\mathbb{E}|\chi_n(X_r)-X_r|^2\Big)dr.
\end{align*}
Therefore
\begin{align*}
&\qquad(2\beta C_1+1)\mathbb{E}|X_t^{(n)}-X_t|^2+2\mathbb{E}|\chi_n(X_t)-X_t|^2\\
&\leq 2\mathbb{E}|\chi_n(X_t)-X_t|^2+C\int_0^t(L_1(r)+1)\gamma\Big((2\beta C_1+1)\mathbb{E}|X_r^{(n)}-X_r|^2
+2\mathbb{E}|\chi_n(X_r)-X_r|^2\Big)dr.
\end{align*}

\noindent{Setting $G(t)=\int_{1}^t\frac{ds}{\gamma(s)}$, the Bihari's inequality admits us to have that}
$$
(2\beta C_1+1)\mathbb{E}|X_t^{(n)}-X_t|^2+2\mathbb{E}|\chi_n(X_t)-X_t|^2
\leq G^{-1}\Big[G\Big(2\mathbb{E}|\chi_n(X_t)-X_t|^2\Big)+C(L_1(T)+1)T\Big].
$$
The fact that $\lim_{n\rightarrow\infty}\chi_n(x)=x$ for $x\in\mathbb{R}^d$ yields that
$\lim_{n\rightarrow\infty}\mathbb{E}|\chi_n(X_t)-X_t|^2=0.$
Recalling the condition $\int_{0+}\frac{ds}{\gamma(s)}=\infty$, we can conclude that
$$
G\Big(2\mathbb{E}|\chi_n(X_t)-X_t|^2\Big)+C(L_1(T)+1)T\rightarrow-\infty, \quad n\rightarrow\infty.
$$
On the other hand, because $G$ is a strictly increasing function, then we obtain that $G$ has an inverse function which is strictly increasing, and $G^{-1}(-\infty)=0$. That is
$$
G^{-1}\Big[G\Big(2\mathbb{E}|\chi_n(X_t)-X_t|^2\Big)+C(L_1(T)+1)T\Big]\rightarrow0, \quad n\rightarrow\infty.
$$
Consequently, we have
$$
\lim_{n\rightarrow\infty}\mathbb{E}|X_t^{(n)}-X_t|^2=0.
$$
Furthermore, we can obtain
$$
\lim_{n\rightarrow\infty}\mathbb{E}|\chi_n(X_t^{(n)})-X_t|^2=0.
$$
So, we get that
$$
\lim_{n\rightarrow\infty}\rho^2(\mu_t^{(n)},\mu_t)
\leq\lim_{n\rightarrow\infty}\mathbb{E}|X_t^{(n)}-X_t|^2=0.
$$
Finally, note that for any $t\in[0,T]$
\begin{align*}
K_t^{(n)}-K_t=X_t-X_t^{(n)}
&+\int_0^t[b^{(n)}(s,X_s^{(n)},\mathscr{L}_{X_s^{(n)}})-b(s,X_s,\mathscr{L}_{X_s})]ds\\
&+\int_0^t[\sigma^{(n)}(s,X_s^{(n)},\mathscr{L}_{X_s^{(n)}})-\sigma(s,X_s,\mathscr{L}_{X_s})]dB_s\\
&+\int_{[0,t]\times U_0}[f^{(n)}(s,X_{s-}^{(n)},\mathscr{L}_{X_s^{(n)}},u)-f(s,X_{s-},\mathscr{L}_{X_s},u)]\widetilde{N}(ds,du).
\end{align*}
Hence, by H\"{o}lder inequality and It\^{o}'s isometry, we have
\begin{align*}
&\quad\mathbb{E}|K_t^{(n)}-K_t|^2\\
&\leq4\mathbb{E}|X_t^{(n)}-X_t|^2
+4\mathbb{E}\Big|\int_0^t[b^{(n)}(s,X_s^{(n)},\mathscr{L}_{X_s^{(n)}})-b(s,X_s,\mathscr{L}_{X_s})]ds\Big|^2\\
&\quad+4\mathbb{E}\Big|\int_0^t[\sigma^{(n)}(s,X_s^{(n)},\mathscr{L}_{X_s^{(n)}})-\sigma(s,X_s,\mathscr{L}_{X_s})]dB_s\Big|^2\\
\end{align*}
\begin{equation}\label{sec4-eq4.7}
\begin{split}
&\quad+4\mathbb{E}\Big|\int_{[0,t]\times U_0}[f^{(n)}(s,X_{s-}^{(n)},\mathscr{L}_{X_s^{(n)}},u)-f(s,X_{s-},\mathscr{L}_{X_s},u)]\widetilde{N}(ds,du)\Big|^2\\
&\leq4\mathbb{E}|X_t^{(n)}-X_t|^2
+4t\mathbb{E}\int_0^t|b^{(n)}(s,X_s^{(n)},\mathscr{L}_{X_s^{(n)}})-b(s,X_s,\mathscr{L}_{X_s})|^2ds\\
&\quad+4\mathbb{E}\int_0^t|\sigma^{(n)}(s,X_s^{(n)},\mathscr{L}_{X_s^{(n)}})-\sigma(s,X_s,\mathscr{L}_{X_s})|^2ds\\
&\quad+4\mathbb{E}\int_{[0,t]\times U_0}|f^{(n)}(s,X_{s}^{(n)},\mathscr{L}_{X_s^{(n)}},u)-f(s,X_{s},\mathscr{L}_{X_s},u)|^2v(du)ds.
\end{split}
\end{equation}
By Assumption \ref{sec2-ass2.9}, we derive that
\begin{align*}
&\quad\mathbb{E}|K_t^{(n)}-K_t|^2\\
&\leq4\mathbb{E}|X_t^{(n)}-X_t|^2
+4(T+1)\mathbb{E}\int_0^tL_1(s)
\kappa\Big(\beta|\chi_{n}(X_{s}^{(n)})-X_{s}|^2+\rho^2(\mathscr{L}_{X_s^{(n)}},\mathscr{L}_{X_s})\Big)ds\\
&\quad+C\mathbb{E}\int_0^tL_1(s)
\varphi\Big(\beta|\chi_{n}(X_{s}^{(n)})-X_{s}|^2+\rho^2(\mathscr{L}_{X_s^{(n)}},\mathscr{L}_{X_s})\Big)ds\\
&\leq4\mathbb{E}|X_t^{(n)}-X_t|^2
+C(T+1)\int_0^tL_1(s)\gamma\Big(\beta\mathbb{E}|\chi_{n}(X_s^{(n)})-X_s|^2+\mathbb{E}|X_s^{(n)}-X_s|^2\Big)ds.\\
\end{align*}
Let $n\rightarrow\infty$, by dominated convergence theorem, we have
$$
\lim_{n\rightarrow\infty}|K_t^{(n)}-K_t|^2=0.
$$
Taking the limit on both sides of \eqref{sec4-eq4.6}, by dominated convergence theorem, one can still obtain \eqref{sec4-eq4.4}.

\noindent{{\bf Step three.} We prove \eqref{sec4-eq4.1} holds.}

\noindent{By {\bf Step two}, we know that $\bar{h}(t,x)$ is $C^1$ in $t$ and $C^2$ in $x$. Applying the classical It\^{o} formula, we obtain that}
\begin{align*}
&\quad dh(t,X_t,\mu_t)=d\bar{h}(t,X_t)\\
&=\partial_t\bar{h}(t,X_t)dt-\langle\partial_x\bar{h}(t,X_t),dK_t\rangle
+\langle\partial_x\bar{h}(t,X_t),b(t,X_t,\mu_t)\rangle dt\\
&\quad+\langle\partial_x\bar{h}(t,X_t),\sigma(t,X_t,\mu_t)dB_t\rangle
+\frac12tr\big((\sigma\sigma^*)(t,X_t,\mu_t)\partial_x^2\bar{h}(t,X_t)\big)dt\\
&\quad+\int_{\mathbb{U}_0}[\bar{h}(t,X_{t}+f(t,X_{t},\mu_t,u))-\bar{h}(t,X_t)-\langle f(t,X_t,\mu_t,u),\partial_x\bar{h}(t,X_t)\rangle]v(du)dt\\
&\quad+\int_{\mathbb{U}_0}[\bar{h}(t,X_{t-}+f(t,X_{t-},\mu_t,u))-\bar{h}(t,X_{t-})]\widetilde{N}(dt,du)\\
&=\partial_th(t,x,\mu)|_{x=X_t,\mu=\mu_t}dt+\partial_th(s,x,\mu_t)|_{s=t,x=X_t}dt-\langle\partial_xh(t,X_t,\mu_t),dK_t\rangle\\
&\quad+\langle\partial_xh(t,X_t,\mu_t),b(t,X_t,\mu_t)\rangle dt
+\langle\partial_xh(t,X_t,\mu_t),\sigma(t,X_t,\mu_t)dB_t\rangle\\
&\quad+\frac12tr\big((\sigma\sigma^*)(t,X_t,\mu_t)\partial_x^2h(t,X_t,\mu_t)\big)dt\\
&\quad+\int_{\mathbb{U}_0}[h(t,X_{t-}+f(t,X_{t-},\mu_t,u),\mu_t)-h(t,X_{t-},\mu_t)]\widetilde{N}(dt,du)\\
&\quad+\int_{\mathbb{U}_0}[h(t,X_{t}+f(t,X_{t},\mu_t,u),\mu_t)-h(t,X_t,\mu_t)-\langle f(t,X_t,\mu_t,u),\partial_xh(t,X_t,\mu_t)\rangle]v(du)dt.
\end{align*}
Inserting \eqref{sec4-eq4.4} in the above equation, we have
\begin{equation}\label{sec4-eq4.9}
\begin{split}
&\qquad dh(t,X_t,\mu_t)\\
&=\partial_th(t,x,\mu)|_{x=X_t,\mu=\mu_t}dt-\langle\partial_xh(t,X_t,\mu_t),dK_t\rangle\\
&\quad+\langle\partial_xh(t,X_t,\mu_t),b(t,X_t,\mu_t)\rangle dt
+\langle\partial_xh(t,X_t,\mu_t),\sigma(t,X_t,\mu_t)dB_t\rangle\\
&\quad+\int_{\mathbb{U}_0}[h(t,X_{t-}+f(t,X_{t-},\mu_t,u),\mu_t)-h(t,X_{t-},\mu_t)]\widetilde{N}(dt,du)\\
&\quad+\frac12tr\big((\sigma\sigma^*)(t,X_t,\mu_t)\partial_x^2h(t,X_t,\mu_t)\big)dt
-\mathbb{E}\langle\partial_{\mu }h(s,y,\mu_{t})(X_t),dK_t\rangle|_{s=t,y=X_t}\\
&\quad+\int_{\mathbb{R}^d}\langle b(t,y,\mu_{t}),\partial_{\mu}h(t,X_t,\mu_{t})(y)\rangle\mu_{t}(dy)dt\\
&\quad+\frac12\int_{\mathbb{R}^d}tr\Big((\sigma\sigma^*)(t,y,\mu_{t})\partial_y\partial_{\mu}h(t,X_t,\mu_{t})(y)\Big)\mu_{t}(dy)dt\\
&\quad+\int_{\mathbb{U}_0}\int_0^1\int_{\mathbb{R}^d}
\Big\langle\partial_{\mu}h(t,X_t,\mu_{t})(y+\eta f(t,y,\mu_{t},u))\\
&\qquad\qquad\qquad\qquad-\partial_{\mu}h(t,X_t,\mu_{t})(y),f(t,y,\mu_{t},u)\Big\rangle\mu_{t}(dy)d\eta v(du)dt\\
&\quad+\int_{\mathbb{U}_0}[h(t,X_{t}+f(t,X_{t},\mu_t,u),\mu_t)-h(t,X_t,\mu_t)-\langle f(t,X_t,\mu_t,u),\partial_xh(t,X_t,\mu_t)\rangle]v(du)dt.\\
\end{split}
\end{equation}
This completes the proof.
\end{proof}

\section{The stability of strong solutions}\label{sec5}
\subsection{The exponential stability of the second moment }\label{sec5.1}
In the subsection, we investigate the exponential stability of the second moment for the strong solution to Equation \eqref{sec1-eq1.1}. Let's first recall the concept of the exponentially stability of the second moment.
\begin{definition}\label{sec5-def5.1}
If there exist a pair of positive constants $\alpha$ and $C$ such that
$$
\mathbb{E}|X_t|^2\leq Ce^{-\alpha t}\mathbb{E}|\xi|^2,\quad t\geq0
$$
for $X_0=\xi$, then the solution $X_.$ of Equation \eqref{sec1-eq1.1} is called exponentially stable in mean square.
\end{definition}

\begin{assumption}\label{sec5-ass5.2}
There exists a function $V:[0,T]\times\mathbb{R}^d\times\mathcal{M}_2(\mathbb{R}^d)\rightarrow\mathbb{R}$ satisfying
\begin{itemize}
\item [(i)]
$V\in C_{b,+}^{1,2,2}([0,T]\times\mathbb{R}^d\times\mathcal{M}_2(\mathbb{R}^d));$
\item [(ii)]
$$\int_{\mathbb{R}^d}(\mathbb{L}_{b,\sigma,f}V(s,x,\mu)+\alpha V(s,x,\mu)+\partial_sV(s,x,\mu))\mu(dx)\leq0,$$
where $\alpha$ is a positive constant;
\item [(iii)]
$$a_1\int_{\mathbb{R}^d}|x|^2\mu(dx)\leq\int_{\mathbb{R}^d}V(s,x,\mu)\mu(dx)\leq a_2\int_{\mathbb{R}^d}|x|^2\mu(dx),$$
where $a_1,a_2$ are two positive constants.
\end{itemize}
\end{assumption}

\begin{theorem}\label{sec5-th5.3}
Assume that Assumptions \ref{sec2-ass2.9}, \ref{sec2-ass2.11} and \ref{sec5-ass5.2} hold, $\xi$ is a $\mathscr{F}_0$-measurable random variable with $\mathbb{E}|\xi|^2<\infty$.  If the strong solution $(X_.,K_.)$ and the Lyapunov function $V$ satisfy for any $t\geq0,$
\begin{equation}\label{sec5-eq5.1}
\begin{split}
\langle\partial_xV(t,X_t,\mathscr{L}_{X_t}),dK_t\rangle+\mathbb{E}\langle(\partial_\mu V)(s,x,\mathscr{L}_{X_t})(X_t),dK_t\rangle|_{s=t,x=X_t}\geq0.
\end{split}
\end{equation}
Then the solution $X_.$ of Equation \eqref{sec1-eq1.1} is exponentially stable in mean square, i.e.
$$
\mathbb{E}|X_t|^2\leq\frac{a_2}{a_1}e^{-\alpha t}\mathbb{E}|\xi|^2, \quad t\geq0.
$$
\end{theorem}
\begin{proof}
Applying the It\^{o} formula \eqref{sec4-eq4.1}, we get
\begin{equation}\label{sec5-eq5.2}
\begin{split}
&\qquad e^{\alpha t}V(t,X_t,\mathscr{L}_{X_t})-V(0,\xi,\mathscr{L}_{\xi})\\
&=\int_0^t\alpha e^{\alpha s}V(s,X_s,\mathscr{L}_{X_s})ds
+\int_0^te^{\alpha s}\partial_sV(s,x,\mu)|_{x=X_s,\mu=\mathscr{L}_{X_s}}ds\\
&\quad-\int_0^te^{\alpha s}\langle\partial_xV(s,X_s,\mathscr{L}_{X_s}),dK_s\rangle
+\int_0^te^{\alpha s}\langle\partial_xV(s,X_s,\mathscr{L}_{X_s}),b(s,X_s,\mathscr{L}_{X_s})\rangle ds\\
&\quad+\int_0^te^{\alpha s}\langle\partial_xV(s,X_s,\mathscr{L}_{X_s}),\sigma(s,X_s,\mathscr{L}_{X_s})dB_s\rangle\\
&\quad+\int_0^t\int_{\mathbb{U}_0}e^{\alpha s}
[V(s,X_{s-}+f(s,X_{s-},\mathscr{L}_{X_s},u),\mathscr{L}_{X_s})-V(s,X_{s-},\mathscr{L}_{X_s})]\widetilde{N}(ds,du)\\
&\quad+\frac12\int_0^te^{\alpha s}
tr\big((\sigma\sigma^*)(s,X_s,\mathscr{L}_{X_s})\partial_x^2V(s,X_s,\mathscr{L}_{X_s})\big)ds\\
&\quad-\int_0^te^{\alpha s}\mathbb{E}\langle(\partial_{\mu }V)
(r,x,\mathscr{L}_{X_s})(X_s),dK_s\rangle|_{r=s,x=X_s}\\
&\quad+\int_0^t\int_{\mathbb{R}^d}e^{\alpha s}
\langle b(s,y,\mathscr{L}_{X_s}),\partial_{\mu}V(s,X_s,\mathscr{L}_{X_s})(y)\rangle\mathscr{L}_{X_s}(dy)ds\\
&\quad+\frac12\int_0^t\int_{\mathbb{R}^d}e^{\alpha s}
tr\big((\sigma\sigma^*)(s,y,\mathscr{L}_{X_s})\partial_y\partial_{\mu}V(s,X_s,\mathscr{L}_{X_s})(y)\big)\mathscr{L}_{X_s}(dy)ds\\
&\quad+\int_0^t\int_{\mathbb{U}_0}\int_0^1\int_{\mathbb{R}^d}
e^{\alpha s}\langle\partial_{\mu}V(s,X_s,\mathscr{L}_{X_s})(y+\eta f(s,y,\mathscr{L}_{X_s},u))\\
&\qquad\qquad\qquad\qquad-\partial_{\mu}V(s,X_s,\mathscr{L}_{X_s})(y),f(s,y,\mathscr{L}_{X_s},u)\rangle\mathscr{L}_{X_s}(dy)d\eta v(du)ds\\
&\quad+\int_0^t\int_{\mathbb{U}_0}e^{\alpha s}[V(s,X_s+f(s,X_s,\mathscr{L}_{X_s},u),\mathscr{L}_{X_s})-V(s,X_s,\mathscr{L}_{X_s})\\
&\qquad\qquad\qquad\qquad-\langle f(s,X_s,\mathscr{L}_{X_s},u),\partial_xV(s,X_s,\mathscr{L}_{X_s})\rangle]v(du)ds.\\
\end{split}
\end{equation}
\noindent{By \eqref{sec5-eq5.1}, it follows that}

\begin{equation}\label{sec5-eq5.3}
\begin{split}
&\quad e^{\alpha t}V(t,X_t,\mathscr{L}_{X_t})-V(0,\xi,\mathscr{L}_{\xi})\\
&\leq\int_0^te^{\alpha s}\Big\{\alpha V(s,X_s,\mathscr{L}_{X_s})+\partial_sV(s,X_s,\mathscr{L}_{X_s})
+\langle\partial_xV(s,X_s,\mathscr{L}_{X_s}),b(s,X_s,\mathscr{L}_{X_s})\rangle\\
&\quad+\frac12tr\big((\sigma\sigma^*)(s,X_s,\mathscr{L}_{X_s})\partial_x^2V(s,X_s,\mathscr{L}_{X_s})\big)
+\int_{\mathbb{R}^d}
\langle b(s,y,\mathscr{L}_{X_s}),\partial_{\mu}V(s,X_s,\mathscr{L}_{X_s})(y)\rangle\mathscr{L}_{X_s}(dy)\\
&\quad+\frac12\int_{\mathbb{R}^d}
tr((\sigma\sigma^*)(s,y,\mathscr{L}_{X_s})\partial_y\partial_{\mu}V(s,X_s,\mathscr{L}_{X_s})(y))\mathscr{L}_{X_s}(dy)\\
&\quad+\int_{\mathbb{U}_0}\int_0^1\int_{\mathbb{R}^d}
\langle\partial_{\mu}V(s,X_s,\mathscr{L}_{X_s})(y+\eta f(s,y,\mathscr{L}_{X_s},u))\\
&\qquad\qquad-\partial_{\mu}V(s,X_s,\mathscr{L}_{X_s})(y),f(s,y,\mathscr{L}_{X_s},u)\rangle\mathscr{L}_{X_s}(dy)d\eta v(du)\\
&\quad+\int_{\mathbb{U}_0}[V(s,X_s+f(s,X_s,\mathscr{L}_{X_s},u),\mathscr{L}_{X_s})-V(s,X_s,\mathscr{L}_{X_s})\\
&\qquad\qquad-\langle f(s,X_s,\mathscr{L}_{X_s},u),\partial_xV(s,X_s,\mathscr{L}_{X_s})\rangle]v(du)\Big\}ds\\
&\quad+\int_0^te^{\alpha s}\langle\partial_xV(s,X_s,\mathscr{L}_{X_s}),\sigma(s,X_s,\mathscr{L}_{X_s})dB_s\rangle\\
&\quad+\int_0^t\int_{\mathbb{U}_0}e^{\alpha s}
[V(s,X_{s-}+f(s,X_{s-},\mathscr{L}_{X_s},u),\mathscr{L}_{X_s})-V(s,X_{s-},\mathscr{L}_{X_s})]\widetilde{N}(ds,du)\\
&=\int_0^te^{\alpha s}(\alpha V(s,X_s,\mathscr{L}_{X_s})+\partial_sV(s,X_s,\mathscr{L}_{X_s})
+\mathbb{L}_{b,\sigma,f}V(s,X_s,\mathscr{L}_{X_s}))ds\\
&\quad+\int_0^te^{\alpha s}\langle\partial_xV(s,X_s,\mathscr{L}_{X_s}),\sigma(s,X_s,\mathscr{L}_{X_s})dB_s\rangle\\
&\quad+\int_0^t\int_{\mathbb{U}_0}e^{\alpha s}
[V(s,X_{s-}+f(s,X_{s-},\mathscr{L}_{X_s},u),\mathscr{L}_{X_s})-V(s,X_{s-},\mathscr{L}_{X_s})]\widetilde{N}(ds,du).\\
\end{split}
\end{equation}

\noindent{Taking the expectation on both sides of the above equality, by Assumption \ref{sec5-ass5.2} we obtain}
\begin{align*}
&\quad e^{\alpha t}\mathbb{E}V(t,X_t,\mathscr{L}_{X_t})-\mathbb{E}V(0,\xi,\mathscr{L}_{\xi})\\
&\leq\mathbb{E}\int_0^te^{\alpha s}(\alpha V(s,X_s,\mathscr{L}_{X_s})+\partial_sV(s,X_s,\mathscr{L}_{X_s})
+\mathbb{L}_{b,\sigma,f}V(s,X_s,\mathscr{L}_{X_s}))ds\leq0.
\end{align*}
Thus,
$$
\mathbb{E}V(t,X_t,\mathscr{L}_{X_t})\leq e^{-\alpha t}\mathbb{E}V(0,\xi,\mathscr{L}_{\xi}).
$$
At last, by Assumption \ref{sec5-ass5.2}, we have
$$
a_1\mathbb{E}|X_t|^2\leq\mathbb{E}V(t,X_t,\mathscr{L}_{X_t})
\leq e^{-\alpha t}\mathbb{E}V(0,\xi,\mathscr{L}_{\xi})\leq a_2e^{-\alpha t}\mathbb{E}|\xi|^2,
$$
Hence,
$$
\mathbb{E}|X_t|^2\leq\frac{a_2}{a_1}e^{-\alpha t}\mathbb{E}|\xi|^2.
$$
This completes the proof.
\end{proof}

\subsection{The exponentially 2-ultimate boundedness }\label{sec5.2}
In the subsection, we study the exponentially 2-ultimate boundedness for the strong solution to Equation \eqref{sec1-eq1.1}. We firstly present the following concept of the exponentially 2-ultimate boundedness.
\begin{definition}\label{sec5-def5.4}
If there exist positive constants $M,\lambda,W$ such that
$$
\mathbb{E}|X_t|^2\leq Me^{-\lambda t}\mathbb{E}|\xi|^2+W,\quad t\geq0,
$$
for $X_0=\xi$, then the solution $X_.$  of Equation \eqref{sec1-eq1.1} is called exponentially 2-ultimately bounded.
\end{definition}

\begin{assumption}\label{sec5-ass5.5}
There exists a function $V:[0,T]\times\mathbb{R}^d\times\mathcal{M}_2(\mathbb{R}^d)\rightarrow\mathbb{R}$ satisfying
\begin{itemize}
\item [(i)]
$V\in C_{b,+}^{1,2,2}([0,T]\times\mathbb{R}^d\times\mathcal{M}_2(\mathbb{R}^d));$
\item [(ii)]
$$\int_{\mathbb{R}^d}(\mathbb{L}_{b,\sigma,f}V(s,x,\mu)+\alpha V(s,x,\mu)+\partial_sV(s,x,\mu))\mu(dx)
\leq N_1,$$
where $\alpha>0$, $N_1\geq0$ are constants;
\item [(iii)]
$$a_1\int_{\mathbb{R}^d}|x|^2\mu(dx)-N_2\leq\int_{\mathbb{R}^d}V(s,x,\mu)\mu(dx)\leq a_2\int_{\mathbb{R}^d}|x|^2\mu(dx)+N_3,$$
where $N_2,N_3 \geq0$ are constants.
\end{itemize}
\end{assumption}

\begin{theorem}\label{sec5-th5.6}
 Suppose that Assumptions \ref{sec2-ass2.9}, \ref{sec2-ass2.11} and \ref{sec5-ass5.5} hold, $\xi$ is a $\mathscr{F}_0$-measurable random variable with $\mathbb{E}|\xi|^2<\infty$.  If the strong solution $(X_.,K_.)$ and the Lyapunov function $V$ satisfy for any $t\geq0$
\begin{equation}\label{sec5-eq5.4}
\begin{split}
\langle\partial_xV(t,X_t,\mathscr{L}_{X_t}),dK_t\rangle+\mathbb{E}\langle(\partial_\mu V)(s,x,\mathscr{L}_{X_t})(X_t),dK_t\rangle|_{s=t,x=X_t}\geq0,
\end{split}
\end{equation}
then the solution $X_.$  of Equation \eqref{sec1-eq1.1} is exponentially 2-ultimately bounded, i.e.
$$
\mathbb{E}|X_t|^2\leq\frac{a_2}{a_1}e^{-\alpha t}\mathbb{E}|\xi|^2+\frac{\alpha(N_2+N_3)+N_1}{\alpha a_1},\quad t\geq0.
$$
Similar to the proof in Theorem \ref{sec5-th5.3}, we can easily obtain the above result. Thus we omit the proof.
\end{theorem}

\subsection{The almost surely asymptotic stability for the strong solution}\label{sec5.3}
In the subsection, we require that $X_0=\xi$ is non-random and consider the almost
 surely asymptotic stability of the strong solution for Equation \eqref{sec1-eq1.1}. To begin with, we recall the concept of the almost surely asymptotic stability.
\begin{definition}\label{sec5-def5.7}
If for $X_0=\xi$, it holds that
$$
\mathbb{P}\Big\{\lim_{t\rightarrow\infty}|X_t|=0\Big\}=1,
$$
we say that the solution $X_.$  of Equation \eqref{sec1-eq1.1}  is almost surely asymptotically stable.
\end{definition}

Next we introduce a function class. Let $\Upsilon$ denote the family of functions $\vartheta:\mathbb{R}_+\rightarrow\mathbb{R}_+$, which are continuous, strictly increasing, and $\vartheta(0)=0$. And $\Upsilon_\infty$ means the family of functions $\vartheta\in\Upsilon$ with $\vartheta(x)\rightarrow\infty$ as $x\rightarrow\infty$. Then we present some assumption.
\begin{assumption}\label{sec5-ass5.8}
There exists a function $V:[0,T]\times\mathbb{R}^d\times\mathcal{M}_2(\mathbb{R}^d)\rightarrow\mathbb{R}$ satisfying
\begin{itemize}
\item [(i)]
$V\in C_{b,+}^{1,2,2;1}([0,T]\times\mathbb{R}^d\times\mathcal{M}_2(\mathbb{R}^d));$
\item [(ii)]
$\mathbb{L}_{b,\sigma,f}V(s,x,\mu)+\alpha V(s,x,\mu)+\partial_sV(s,x,\mu)\leq 0,$
where $\alpha>0$ is a constant;
\item [(iii)]
$\gamma_1(|x|)\leq V(s,x,\mu)\leq\gamma_2(|x|),$ where $\gamma_1,\gamma_2\in\Upsilon_\infty$.
\end{itemize}
\end{assumption}

\begin{theorem}\label{sec5-th5.9}
 Suppose that Assumptions \ref{sec2-ass2.9}, \ref{sec2-ass2.11} and \ref{sec5-ass5.8} hold. If the strong solution $(X_.,K_.)$ and the Lyapunov function $V$ satisfy for any $t\geq0$
\begin{equation}\label{sec5-eq5.5}
\begin{split}
\langle\partial_xV(t,X_t,\mathscr{L}_{X_t}),dK_t\rangle+\mathbb{E}\langle(\partial_\mu V)(s,x,\mathscr{L}_{X_t})(X_t),dK_t\rangle|_{s=t,x=X_t}\geq0,
\end{split}
\end{equation}
then the solution $X_.$  of Equation \eqref{sec1-eq1.1}  is almost surely asymptotically stable, i.e.
$$
\mathbb{P}\Big\{\lim_{t\rightarrow\infty}|X_t|=0\Big\}=1.
$$
\end{theorem}

\begin{proof}
 Above all, since under Assumptions \ref{sec2-ass2.9} and \ref{sec2-ass2.11}, Equation \eqref{sec1-eq1.1} has a unique strong solution $(X_.,K_.)$ with the initial value $(\xi,0)$. Set $\tau_n:=inf\{t\geq0,|X_t|>n\}$. By It\^{o}'s formula, we have
\begin{align*}
|X_{t\wedge\tau_n}-\xi|^2&=-2\int_0^{t\wedge\tau_n}\langle X_s-\xi,dK_s\rangle
+2\int_0^{t\wedge\tau_n}\langle X_s-\xi,b(s,X_s,\mathscr{L}_{X_s})\rangle ds\\
&\quad+2\int_0^{t\wedge\tau_n}\langle X_s-\xi,\sigma(s,X_s,\mathscr{L}_{X_s})dB_s\rangle
+\int_0^{t\wedge\tau_n}\|\sigma(s,X_s,\mathscr{L}_{X_s})\|^2ds\\
&\quad+\int_{[0,t\wedge\tau_n]\times U_0}|f(s,X_{s},\mathscr{L}_{X_s},u)|^2v(du)ds\\
&\quad+\int_{[0,t\wedge\tau_n]\times U_0}
(|X_{s-}-\xi+f(s,X_{s-},\mathscr{L}_{X_s},u)|^2-|X_{s-}-\xi|^2)\widetilde{N}(ds,du).\\
\end{align*}

\noindent{Therefore,}
\begin{align*}
&\quad\mathbb{E}(\sup_{0\leq s\leq t}|X_{s\wedge\tau_n}-\xi|^2)\\
&\leq\mathbb{E}\int_0^{t\wedge\tau_n}|X_s-\xi|^2ds
+\mathbb{E}\int_0^{t\wedge\tau_n}(|b(s,X_s,\mathscr{L}_{X_s})|^2+\|\sigma(s,X_s,\mathscr{L}_{X_s})\|^2)ds\\
&\quad+\mathbb{E}\int_{[0,t\wedge\tau_n]\times U_0}|f(s,X_{s},\mathscr{L}_{X_s},u)|^2v(du)ds
+2\mathbb{E}\Big(\sup_{0\leq s\leq t}\int_0^{s\wedge\tau_n}\langle X_r-\xi,\sigma(r,X_r,\mathscr{L}_{X_r})dB_r\rangle\Big)\\
\end{align*}
\begin{equation}\label{sec5-eq5.6}
\begin{split}
&\quad+2\mathbb{E}\Big(\sup_{0\leq s\leq t}\int_{[0,s\wedge\tau_n]\times U_0}\langle X_{r-}-\xi,f(r,X_{r-},\mathscr{L}_{X_r},u)\rangle\widetilde{N}(dr,du)\Big)\\
&\quad+\mathbb{E}\Big(\sup_{0\leq s\leq t}\int_{[0,s\wedge\tau_n]\times U_0}
|f(r,X_{r-},\mathscr{L}_{X_r},u)|^2\widetilde{N}(dr,du)\Big).\\
\end{split}
\end{equation}

\noindent{Burkholder-Davis-Gundy inequality implies}
\begin{equation}\label{sec5-eq5.7}
\begin{split}
&\quad2\mathbb{E}\Big(\sup_{0\leq s\leq t}\int_0^{s\wedge\tau_n}\langle X_r-\xi,\sigma(r,X_r,\mathscr{L}_{X_r})dB_r\rangle\Big)\\
&\leq C\mathbb{E}\Big(\int_0^{t\wedge\tau_n}|X_s-\xi|^2\|\sigma(s,X_s,\mathscr{L}_{X_s})\|^2ds\Big)^{\frac12}\\
&\leq \frac13\mathbb{E}(\sup_{0\leq s\leq t}|X_{s\wedge\tau_n}-\xi|^2)
+C\mathbb{E}\int_0^{t\wedge\tau_n}\|\sigma(s,X_s,\mathscr{L}_{X_s})\|^2ds.
\end{split}
\end{equation}

\noindent{With the help of Lemma \ref{sec2-lem2.1}, we deduce}
\begin{equation}\label{sec5-eq5.8}
\begin{split}
&\quad2\mathbb{E}\Big(\sup_{0\leq s\leq t}\int_{[0,s\wedge\tau_n]\times U_0}\langle X_{r-}-\xi,f(r,X_{r-},\mathscr{L}_{X_r},u)\rangle\widetilde{N}(dr,du)\Big)\\
&\leq C\mathbb{E}\Big(\int_{[0,t\wedge\tau_n]\times U_0}|\langle X_{s-}-\xi,f(s,X_{s-},\mathscr{L}_{X_s},u)\rangle|^2
 N(ds,du)\Big)^{\frac12}\\
&\leq C\mathbb{E}\Big(\sup_{0\leq s\leq t}|X_{s\wedge\tau_n}-\xi|^2\int_{[0,t\wedge\tau_n]\times U_0} |f(s,X_{s-},\mathscr{L}_{X_s},u)|^2 N(ds,du)\Big)^{\frac12}\\
&\leq\frac13\mathbb{E}(\sup_{0\leq s\leq t}|X_{s\wedge\tau_n}-\xi|^2)+C\mathbb{E}\int_{[0,t\wedge\tau_n]\times U_0}|f(s,X_{s},\mathscr{L}_{X_s},u)|^2v(du)ds,
\end{split}
\end{equation}
\noindent{and}
\begin{equation}\label{sec5-eq5.9}
\begin{split}
&\quad\mathbb{E}\Big(\sup_{0\leq s\leq t}\int_{[0,s\wedge\tau_n]\times U_0}
|f(r,X_{r-},\mathscr{L}_{X_r},u)|^2\widetilde{N}(dr,du)\Big)\\
&\leq C\mathbb{E}\Big\{\sum_{s\in D_{p(s)},s\leq t\wedge\tau_n}
|f(s,X_{s-},\mathscr{L}_{X_s},p(s))|^4\Big\}^{\frac12}\\
&\leq C\mathbb{E}\sum_{s\in D_{p(s)},s\leq t\wedge\tau_n}
|f(s,X_{s-},\mathscr{L}_{X_s},p(s))|^2\\
&= C\mathbb{E}\int_{[0,t\wedge\tau_n]\times U_0}|f(s,X_{s},\mathscr{L}_{X_s},u)|^2v(du)ds.
\end{split}
\end{equation}

\noindent{Gathering all the above estimates \eqref{sec5-eq5.7}-\eqref{sec5-eq5.9} into \eqref{sec5-eq5.6}, we obtain}
\begin{equation}\label{sec5-eq5.10}
\begin{split}
&\quad\mathbb{E}(\sup_{0\leq s\leq t}|X_{s\wedge\tau_n}-\xi|^2)\\
&\leq3\mathbb{E}\int_0^{t\wedge\tau_n}|X_s-\xi|^2ds
+C\mathbb{E}\int_0^{t\wedge\tau_n}(|b(s,X_s,\mathscr{L}_{X_s})|^2+\|\sigma(s,X_s,\mathscr{L}_{X_s})\|^2)ds\\
&\quad+C\mathbb{E}\int_{[0,t\wedge\tau_n]\times U_0}|f(s,X_{s},\mathscr{L}_{X_s},u)|^2v(du)ds.
\end{split}
\end{equation}

\noindent{By Assumption \ref{sec2-ass2.9}, we have}
\begin{align*}
&\quad\mathbb{E}(\sup_{0\leq s\leq t}|X_{s\wedge\tau_n}-\xi|^2)\\
&\leq3\mathbb{E}\int_0^{t\wedge\tau_n}|X_s-\xi|^2ds
+C\mathbb{E}\int_0^{t\wedge\tau_n}L_2(s)(1+|X_s|^2+\|\mathscr{L}_{X_s}\|_2^2)ds\\
&\leq6\int_0^{t\wedge\tau_n}(\mathbb{E}|X_s|^2+\mathbb{E}|\xi|^2)ds
+CL_2(T)\int_0^{t\wedge\tau_n}(1+2\mathbb{E}|X_s|^2)ds\\
&\leq6\int_0^{t\wedge\tau_n}(n^2+\mathbb{E}|\xi|^2)ds
+CL_2(T)\int_0^{t\wedge\tau_n}(1+2n^2)ds\\
&\leq C(t\wedge\tau_n)\leq Ct.
\end{align*}
Using the Chebyshev inequality, we have for any $\lambda>0$,
\begin{equation}\label{sec5-eq5.11}
\begin{split}
\mathbb{P}\Big\{\sup_{0\leq s\leq t}|X_{s\wedge\tau_n}-\xi|>\lambda\Big\}\leq\frac{Ct}{\lambda^2}.
\end{split}
\end{equation}
Finally, we follow up the line in (\cite{X. Ding}, Theorem 5.2) and apply \eqref{sec5-eq5.11} to derive
$$
\mathbb{P}\Big\{\lim_{t\rightarrow\infty}|X_t|=0\Big\}=1.
$$
This completes the proof.
\end{proof}

\bigskip
$\begin{array}{cc}
\begin{minipage}[t]{1\textwidth}
{\bf Guangjun Shen}\\
Department of Mathematics, Anhui Normal University, Wuhu 241002, China\\
\texttt{gjshen@163.com}
\end{minipage}
\hfill
\end{array}$

$\begin{array}{cc}
\begin{minipage}[t]{1\textwidth}
{\bf Jie Xiang}\\
Department of Mathematics, Anhui Normal University, Wuhu 241002, China\\
\texttt{jiexiangahnu@163.com}
\end{minipage}
\hfill
\end{array}$

$\begin{array}{cc}
\begin{minipage}[t]{1\textwidth}
{\bf Jiang-Lun Wu}\\
Department of Mathematics, Computational Foundry,
 Swansea University \\ Swansea, SA1 8EN, UK\\
\texttt{j.l.wu@swansea.ac.uk}
\end{minipage}
\hfill
\end{array}$

\end{document}